\documentclass[12pt,reqno]{amsart}

\usepackage{amssymb,amsmath,amsbsy,eucal,amsfonts,mathrsfs,latexsym}
\usepackage{graphicx,verbatim,psfrag,epsfig,enumerate}
\usepackage{stmaryrd}
\usepackage{lscape}
\usepackage{url}
\usepackage{subfig}
\usepackage{multirow}
\usepackage{color}

\newtheorem{theorem}{Theorem}[section]
\newtheorem{lemma}[theorem]{Lemma}
\newtheorem{proposition}[theorem]{Proposition}

\theoremstyle{definition}
\newtheorem{definition}[theorem]{Definition}

\newtheorem{algorithm}{Algorithm}[section]
\newtheorem{remark}{{\it Remark}}[section]

\numberwithin{equation}{section}

\newcommand{\lb}{\llbracket}
\newcommand{\rb}{\rrbracket}
\newcommand{\Lb}{\{\hspace{-4.0pt}\{}
\newcommand{\Rb}{\}\hspace{-4.0pt}\}}
\newcommand{\vertiii}[1]{{\left\vert\kern-0.25ex\left\vert\kern-0.25ex\left\vert #1 
    \right\vert\kern-0.25ex\right\vert\kern-0.25ex\right\vert}}

\newcommand{\Bn}{{\boldsymbol{n}}}

\newcommand{\Bq}{{\boldsymbol{q}}}

\newcommand{\Bv}{{\boldsymbol{v}}}

\newcommand{\BV}{{\boldsymbol{V}}}

\newcommand{\Ce}{{\mathcal E}}
\newcommand{\Cf}{{\mathcal E}}

\newcommand{\Cj}{{\mathcal J}}

\newcommand{\Cs}{{\mathcal S}}
\newcommand{\Ct}{{\mathcal T}}

\definecolor{red}{rgb}{1,0,0}

\definecolor{blue}{rgb}{0,0,1}

\DeclareMathOperator*{\argmin}{arg\,min}

\usepackage{mathtools,xcolor}
\usepackage[colorlinks=true,linkcolor=blue]{hyperref}
\usepackage{cleveref}
\crefname{equation}{}{}
\usepackage{algorithm}
\usepackage{algorithmic}

\newtheorem{pps}[theorem]{Proposition}

\newcommand{\bN}{\mathbb{N}}
\newcommand{\bX}{\mathbb{X}}
\newcommand{\bR}{\mathbb{R}}
\newcommand{\bS}{\mathbb{S}}
\newcommand{\bW}{\mathbb{W}}

\newcommand{\norm}[1]{{\left|#1\right|}}
\newcommand{\normm}[1]{{\left\|#1\right\|}}
\newcommand{\nj}[2]{{\left\langle{#1},{#2}\right\rangle}}
\newcommand{\kh}[1]{{\left(#1\right)}}
\newcommand{\zkh}[1]{{\left[#1\right]}}
\newcommand{\zzkh}[1]{{\left\llbracket#1\right\rrbracket}}
\newcommand{\dkh}[1]{{\left\{#1\right\}}}
\newcommand{\jz}[1]{\begin{bmatrix}#1\end{bmatrix}}

\DeclareMathOperator{\Id}{{\it I}}
\DeclareMathOperator{\prox}{prox}
\DeclareMathOperator{\Proj}{Proj}
\DeclareMathOperator{\sgn}{Sgn}

\newcommand{\FPPA}{FP\textsuperscript{2}\!A}

\begin{document}

\title{A Non-gradient DG method for second-order Elliptic Equations in the Non-divergence Form}

\thanks{W. Qiu's research is partially supported by the Research Grants Council of the Hong Kong Special
    Administrative Region, China. (Project Nos. CityU 11302219, CityU 11300621). K. Shi is partially 
    supported by National Science Foundation (Award Number: 2012235). Y. Xu is supported in part by the US National Science Foundation under grants DMS-1912958 and DMS-220838, and by the US National Institutes of Health under grant R21CA263876.
 As a convention the names of the authors 
    are alphabetically ordered. All authors contributed equally in this article.
}

\author{Weifeng Qiu}
\address{Department of Mathematics, City University of Hong Kong, 83 Tat Chee Avenue, Hong Kong, China.}
\email{weifeqiu@cityu.edu.hk}

\author{Jin Ren}
\address{Department of Mathematics and Statistics, Old Dominion University, Norfolk, VA 23529, USA.}
\email{jren@odu.edu}

\author{Ke Shi}
\address{Department of Mathematics and Statistics, Old Dominion University, Norfolk, VA 23529, USA.}
\email{kshi@odu.edu}

\author{Yuesheng Xu}
\address{Department of Mathematics and Statistics, Old Dominion University, Norfolk, VA 23529, USA.}
\email{y1xu@odu.edu}

\begin{abstract}
$L^1$ based optimization is widely used in image denoising, machine learning and related applications. One of the main features of such approach is that it naturally provide a sparse structure in the numerical solutions. In this paper, we study an $L^1$ based mixed DG method for second-order elliptic equations in the non-divergence form. The elliptic PDE in nondivergence form arises in the linearization of fully nonlinear PDEs. Due to the nature of the equations, classical finite element methods based on variational forms can not be employed directly. In this work, we propose a new optimization scheme coupling the classical DG framework with recently developed $L^1$ optimization technique. Convergence analysis in both energy norm and $L^{\infty}$ norm are obtained under weak regularity assumption. 
Such $L^1$ models are nondifferentiable and therefore invalidate traditional gradient methods. Therefore all existing gradient based solvers are no longer feasible under this setting. To overcome this difficulty, we characterize solutions of $L^1$ optimization as fixed-points of proximity equations and utilize matrix splitting technique to obtain a class of fixed-point proximity algorithms with convergence analysis. Various numerical examples are displayed to illustrate the numerical solution has sparse structure with careful choice of the bases of the finite dimensional spaces.
Numerical examples in both smooth and nonsmooth settings are provided to validate the theoretical results.    
\end{abstract}



\maketitle

\section{Introduction}

{{One of the main features of using $L^1$ norm based optimizations is that the approximation naturally has a sparse representation under specified transformations or bases. This feature is particularly important in large scale computations. For instance, in multi-physics applications, one needs to use the solution from one model as inputs/parameters for the subsequent physical model. 
A typical challenge lays in the transition of computation between models with different scales. 
Due to the curse of dimensionality, the computational complexity grows exponantially as the number of independent inputs/paramters grow linearly. In other words, too many independent inputs will lead to extremely large system to solve. Therefore, it is desirable to develop {\it sparsity promoting} algorithms such that the input, the solution from the foremal model can be represented by a small number of basis functions. }}

{{In the area of numerical PDEs, numerical schemes based on $L^1$ stablization are less explored. The main challenge is that the $L^1$ norm based scheme will result in minimizing an non-differentiable functional. In \cite{Guermond2004}, the author studied a finite element method for solving first order PDEs under the $L^1$ norm. The author overcome the difficulty of minimizing the non-differentiable functional by regularizing the functional with a sequence of differentiable minimization problems. }}

{{In this paper, we propose a numerical scheme with two main features: (1) using a sparsity promoting functional in the minimization scheme and adopting fixed-point proximity algorithm \cite{micchelli2011proximity,micchelli2013proximity,li2015multi,RenInexact} with exact convergence analysis of non-smooth optimization; (2) A sparsity promoting base of the finite element space \cite{Chen1999construction,Chen2002Fast,Micchelli1994Using,Micchelli1997Reconstruction}. The first condition is met by using $L^1$ stabilization while the second is met by using DG finite elements. In \cite{Chen15} the authors systematically studied the construction of multiscale bases under the DG setting. It is worth to mention that the current study is an ongoing effort to develop new numerical schemes such that the numerical solutions with sparse structure in the finite dimensional space. As a step-stone, we consider the following elliptic equations in non-divergence form:}}

\begin{align}
\label{non_div_governing}
A:D^2 u &= f \quad \text{in} \quad \Omega, \\
u & = 0, \quad \text{on} \quad \partial \Omega. 
\end{align}

Here the domain $\Omega \in \mathbb{R}^d$ $(d=2,3)$ is an open and bounded polytope and the coefficient matrix $A =[A_{ij}] \in [L^{\infty}(\Omega)]^{d\times d}$ is a symmetric positive definite matrix with eigenvalues bounded between $(0, \Lambda)$ with $\Lambda > 0$. $D^2 u = [\frac{\partial^2 u}{\partial x_i \partial x_j}]$ denotes the Hessian matrix of $u$ and $A:D^2u := \sum_{ij} A_{ij} \frac{\partial^2 u}{\partial x_i \partial x_j}$. Note that $A$ is not necessarily differentiable. Indeed if $A \in [C^1(\Omega)]^{d\times d}$ we may recast the above system into a divergence form \cite{QiuZhang20}. Hence, the standard steady-state diffusion problem is a special case of \eqref{non_div_governing}. 

The PDEs in non-divergence form like \eqref{non_div_governing} arises in the linearization of fully nonlinear PDEs such as stochastic control problems \cite{FlemingSoner06}, Monge-Amp\'ere equations \cite{NeilanSalgadoZhang20}, and the fully nonlinear Hamilton-Jacobi-Bellman equations \cite{JensonSmears12}. One of the main challenges around this problem is the fact that the coefficient matrix $A(x)$ in general is not smooth in many practical settings. In this case, the standard weak variational formulation of the PDE is no longer available. Indeed, analytical results on the PDE level such as the existence and uniqueness of solutions for \eqref{non_div} are often characterized in the classical or strong sense. 

Due to the possible non-smooth cofficient matrix, developing standard numerical methods for problem \eqref{non_div_governing} with variational form is very difficult. In \cite{LakkisPryer11} the authors study a mixed finite element method based on Galerkin framework by introducing a so-called ``finite element Hessian'' and recently in \cite{LakkisPryer21} they extend the idea with discontinuous Galerkin elements. The analysis of the DG method is provided in \cite{Neilan17}.
In \cite{FengHeningsNeilan17, FengNeilanStephan18} the authors consider two DG type methods based on variational form. The key ingredient for the error estimates is on the discrete Calderon-Zygmond estimates which assumes a full $H^{2,p}$ regularity for the PDE operator: 
\begin{align}
\label{W2p_assumption}
\| w \|_{2,p,\Omega} \le C_{p} \| A: D^{2}w \|_{p,\Omega}, \quad \forall w \in W_{0}^{1,p}(\Omega) \cap W^{2,p}(\Omega).
\end{align}
Here $1 \leq p < \infty$. This regularity assumption holds if the coefficient matrix 
$A \in [C^0(\overline{\Omega})]^{d \times d}$ and $\partial \Omega$ is $C^{1,1}$. 
In contrast, several methods are studied \cite{SmearsSuli13,Gallistl17} for equations with non-smooth coefficient matrix. they considered DG method for discontinuous $A$ with Cord\'es condition and convex polyhedral domain $\Omega$. A primal-dual weak Galerkin method is introduced by Wang and Wang \cite{WangWang18}. Roughly speaking, Cord\'es condition requires that the possibly discontinuous coefficient matrix $A$ is not too far away from identity matrix multiplied by a constant. Consequently the difference between the nondivergence form and a diffusive form is bounded by the $|D^2 u|$ pointwisely.
Recently, a $L^p$-weak Galerkin method for the problem is developed in \cite{CaoWangXu22} based on primal dual weak Galerkin framework. They recast the problem into a Min-Max optimization problem. All the above mentioned works require a full $W^{2,p}$ 
regularity assumption (\ref{W2p_assumption}) for the PDE operator in order to establish the stability of the methods (methods in \cite{SmearsSuli13, Gallistl17, WangWang18} are valid only for $p = 2$).

In a recent work by one of the authors Qiu \cite{QiuZhang20}, they develop a first order system least squares method for \eqref{non_div_governing}. It is worth to mention that for the stability of the method, their method only requires $W^{1,2}$ regularity of the PDE operator:
\begin{align}
\label{H1_assumption}
\| w \|_{1,2,\Omega} \le C \| A: D^{2}w \|_{2,\Omega}, 
\end{align}
for any $w \in \{ v \in W_{0}^{1,2}(\Omega): A:D^{2}v \in L^{2}(\Omega)\}$. All the above mentioned work are formulated and implemented as Galerkin type schemes. Nochetto and Zhang \cite{NochettoZhang18} studied a two-scale method, which is based on the integro-differential approach and focuses on $L^{\infty}$-error estimates. 

In this paper we advocate a different approach to discretize the equation \eqref{non_div_governing} under the minimization framework. Namely, we employ standard DG element for the primary unknown $u$ and its gradient $\nabla u$. We augment the minimization problem with $L^1$-type penalty terms. Comparing with existing finite element methods for solving \eqref{non_div} based on variational formulations, our approach has a few distinct features. First, the discrete problem does not need to use integration by parts which fits the original PDE naturally. No artificial differential operators are needed. Consequently, the error estimates is relatively neat and simple. Secondly, we are able to relax the regularity requirement of the PDE operator from $W^{2,p}$ (\ref{W2p_assumption}) to $W^{1,2}$ (\ref{H1_assumption}). We only need the exact solution of the PDE $u \in W^{2+\delta,2}$ to achieve convergence result in the energy norm comparable with above mentioned existing work. 
Furthermore, we establish pointwise error estimate based on the Aleksandrov-Bakelman-Pucci estimates \cite{GilbargTrudinger1983}, when the domain $\Omega \in \mathbb{R}^d (d=2,3)$ is an open and bounded polytope and the coefficient matrix $A = A(x) \in [L^{\infty}(\Omega)]^{d\times d}$ is a symmetric positive definite matrix with eigenvalues bounded between $(0, \Lambda)$ with $\Lambda > 0$ (we need to assume $A$ is piecewise polynomial if $d = 3$). Finally, conventional methods such as least squares or Galerkin type methods based on variational forms are feasible for PDEs with smooth solutions. Nevertheless, when the PDE operator is lack of strong regularity such as \eqref{non_div_governing} and/or the desired energy functional is not differentiable ($L^1$ for instance), the weak formulation of a PDE is not application  under such settings.


Unlike smooth optimization, the proposed $L^1$ optimization cannot be solved via gradient-based algorithms due to its nondifferentiability. To overcome the difficulty brought by nonsmoothness, here we adopt the framework of Fixed-Point Proximity Algorithms (\FPPA{}) originated from image processing and data science \cite{micchelli2011proximity,micchelli2013proximity,li2015multi,RenInexact} to solve the proposed $L^1$ optimization problem. \FPPA{} first characterizes solutions of $L^1$ optimization to be fixed-points of proximity equations of nonsmooth functions, then a matrix splitting technique provides a class of fixed-point proximity algorithms. The framework of firmly nonexpansive operators are then applied to prove the convergence of \FPPA{}. It is also proven in \cite{li2015multi} that \FPPA{} covers popular algorithms such as first-order primal-dual algorithms and alternating direction method of multipliers.


The rest of the paper is organized as follows. In Section 2, we present the details of the mixed DG method for the problem \eqref{non_div_governing} with $L^1$ stabilization and the main error estimates results. In Section 3 we present the detailed proofs of the error estimstes. In Section 4, we review some preliminary approximation results and present the main error estimates of the method. Given the fact that the scheme is based on $L^1$ stabilization, the implementation of far from trivial, in Section 5, we will present the details on the iterative algorithm for the scheme and present the convergence analysis of the algorithm. Several numerical examples will be performed in Section 5 with concluding remarks in Section 6. 

\section{Mixed formulation for non-divergence form} 
In this section we present the mixed DG method for the problem \eqref{non_div_governing} as a minimization problem. We begin by specializing \eqref{non_div_governing} as a first order system:
\begin{subequations}
\label{non_div}
\begin{align}
\label{non_div_eq1}
\nabla \Bq - u & = 0 \quad \text{in} \quad \Omega, \\ 
\label{non_div_eq2}
A:\nabla \Bq &= f \quad \text{in} \quad \Omega, \\
\label{non_div_eq3}
u & = 0, \quad \text{on} \quad \partial \Omega. 
\end{align}
\end{subequations}

We define the finite element spaces for the numerical scheme for solving \eqref{non_div}. 
We adopt the notation and norms for the spaces as in \cite{ABCD02}. 
We consider a family of conforming triangulations $\Ct_{h}$ made of {{shape-regular simplexes (triangles in 2D, tetrahedra in 3D). Here conforming means that for any two element $K, K' \in \Ct_h$, $\overline{K} \cap \overline{K'}$ is either $\emptyset$, a common vertex, a common edge or a common face of both elements. }} 
Let $\Cf_{h}^{I}$ denote the set of all interior faces of $\Ct_{h}$, and $\Cf_{h}^{B}$ denotes the set of 
all boundary faces. We define $\Cf_{h}:= \Cf_{h}^{I} \cup \Cf_{h}^{B}$. We use $h_K$ to denote the diameter 
of the element $K$, and $h_{F}$ the diameter of the edge $F\in \Cf_{h}^{B}$. The mesh size of $\Ct_{h}$ is defined as 
$h:= \max_{K\in \Ct_{h}}h_{K}$. We denote by $\Bn_{K}$ the unit outward normal vector on $\partial K$. We will drop the sub-index $K$ to denote a generic outer normal vector $\Bn$.
We also introduce the average and jump operators. Let $F = \partial K\cap \partial K^{\prime}$ be 
an interior face shared by $K$ and $K^{\prime}$.
$\phi$ is a generic piecewise smooth function 
(scalar- or vector-valued). We define the average of $\phi$ on $F$ as 
$$\Lb \phi \Rb := \frac{1}{2}(\phi + \phi^{\prime})$$ where $\phi$ and $\phi^{\prime}$ denote the trace of 
$\phi$ from the interior of $K$ and $K^{\prime}$ {{respectively}}. Furthermore, let $w$ be a piecewise smooth function 
and $\Bv$ a piecewise smooth vector-valued field. Analogously, we define the following jumps on $F$:
\begin{align*}
& \lb w \rb  := w\Bn_{K} + w^{\prime} \Bn_{K^{\prime}},  
& \lb \Bv \cdot \Bn \rb := \Bv \cdot \Bn_{K} + \Bv^{\prime}\cdot \Bn_{K^{\prime}}.
\end{align*}
On a boundary face $F= \partial K \cap \partial\Omega$, we set accordingly $$\Lb \phi \Rb := \phi, \quad \lb w\rb := w \Bn, \quad \text{and} \quad \lb \Bv\rb_{N} := \Bv\cdot \Bn.$$ 
Let $P_k(D)$ denote the space of polynomials of degree no more than $k$ over the domain $D$. We define the following spaces for $k \ge 2$:
\begin{align*}
W_h &=\{w \in L^2(\Omega)| w_K \in P_k(K), \forall K \in \Ct_h\}, \\
\BV_h &=\{ \Bv \in [L^2(\Omega)]^d | \Bv_K \in [P_{k-1}(K)]^d, \forall K \in \Ct_h, \}, \\
W^c_h &= W_h \cap H^1_0(\Omega), \quad \BV^c_h = \BV_h \cap H(\text{div};\Omega).
\end{align*}
Finally, we use the standard notation for Sobolev norms: $\|\cdot \|_{k,p,D}$ and $|\cdot|_{k,p}$ are used for the standard $W^{k,p}(D)$ norms and semi-norms, respectively. In addition, for discrete functions we define:
$$
\|w\|_{k,p,\Ct_h}:= \left(\sum_{K \in \Ct_h} \|w\|^p_{k,p,K}\right)^{\frac{1}{p}},
$$
$$
\|w\|_{k,p,\Cf_h}:= \left(\sum_{F \in \Cf_h} \|w\|^p_{k,p,F}\right)^{\frac{1}{p}}, 
$$
$$
\|w\|_{k,p,\Cf^I_h}:= \left(\sum_{F \in \Cf^I_h} \|w\|^p_{k,p,F}\right)^{\frac{1}{p}},
$$
We drop the index $k$ when $k=0$ and we drop both indices $k,p$ for the standard $L^2$ norm. i.e.
$\|w\|_{p,D}:= \|w\|_{0,p, D}$ and $\|w\|_{D}:= \|w\|_{0,2,D}$.

We next describe our numerical scheme.
For $(w,v) \in W_h \times \BV_h$, we introduce the {\em penalty term}
\begin{align}
\label{stab_term}
\Cs(\Bv,w):= (h^{-1}\|\lb \Bv \cdot \Bn \rb\|_{1,\Cf^I_h} + h^{-2}\|\lb w \rb\|_{1,\Cf_h}) 
\end{align}
and the energy functional
$$
\mathcal{J}_h(w,\Bv) := \|A : \nabla \Bv - f\|^2_{\Ct_h} + h^{-2} \|\Bv - \nabla w\|^2_{\Ct_h} + \tau \Cs(\Bv,w),
$$
where the $h$-independent parameter $\tau$ can be chosen as any positive real number bigger or equal to $1$. 
We seek an approximation $(u_h, \Bq_h) \in W_h \times \BV_h$ by solving the minimization problem
\begin{align}
\label{mini_h}
(u_h, \Bq_h) &= \argmin_{(w,v) \in W_h \times \BV_h} \mathcal{J}_h(w,\Bv). 
\end{align}
The objective function of minimization problem \eqref{mini_h} is convex and thus, it has a (global) minimizer. However, since it involves a non-smooth penalty term, standard gradient type algorithm cannot be applied to this problem. Fixed-point iterative algorithms will be developed in section 4 for solving this minimization problem.


\begin{remark}
For non-homogeneous boundary condition case, we should modify the penalty term (\ref{stab_term}) as:
\[
\Cs(\Bv,w) := \|\lb \Bv \cdot \Bn \rb\|_{1,\Cf^I_h} + h^{-2}\|\lb w \rb\|_{1,\Cf^I_h} + h^{-2} \|w - g\|_{1, \Cf_h^B}.
\]
\end{remark}

\begin{remark}
Note that in the energy functional $\mathcal{J}_h$ we can replace the $L^2$ norm with any $L^p$ norm with 
$1 \le p < \infty$; the error analysis should follow almost the same lines as $L^2$ since we do not rely on variational formulations. Nevertheless, we use the $L^2$ norm here for the sake of simplicity in the implementation. As part of our future work, we will develop efficient algorithms to treat general $L^p$ norms in the functional. 
\end{remark}

\begin{remark}
 We believe that the framework can be formulated with the primary variable $u$ only. It should not have substantial difficulty in the analysis of the method. Nevertheless, for the sake of simplicity in the implementation step, we recast the problem into the first order system \eqref{non_div} and approximate $u$ and $\Bq$ together. 
\end{remark}

In this paper, we provide two types of error estimates of the above scheme. We first provide the error estimate with the standard discrete $H^1$ norms. In this case we assume that the elliptic equation has the $H^1$ regularity: 
for any $w \in \{ v \in H_{0}^{1}(\Omega) : A:D^{2}v \in L^{2}(\Omega) \}$, 
\begin{equation}
\label{regularity_NonDiv}
\|w\|_{1,2,\Omega} \leq C \| A:D^{2}w \|_{\Omega}.
\end{equation}

It is worth to mention that the above regularity assumption is weaker than the well-known Cord\'es condition for the full $H^2$ regularity \cite{SmearsSuli13}. Namely, Cord\'{e}s condition requires $\Omega$ must be convex polyhedral domain and $A \in [L^{\infty}(\Omega)]^{d\times d}$ is symmetric positive definite and there exists $\epsilon \in (0,1]$ such that
\[
\frac{\sum_{i,j}a^2_{ij}}{({\bf Tr} A)^2} \le \frac{1}{n-1+\epsilon} \quad \text{a.e. in} \quad \Omega. 
\]
In contrast, the relaxed $H^1$ regularity \eqref{regularity_NonDiv} holds in more general settings. For instance, if $\Omega \in \mathbb{R}^3$ is a polyhedral domain (possibly nonconvex), we can establish \eqref{regularity_NonDiv} with similar assumption as Cond\'{e}s condition as follows:
\begin{proposition}\label{nonconvex_reg}
Let $\Omega \in \mathbb{R}^3$ be a polyhedral domain. If $A \in [L^{\infty}(\Omega)]^{3\times 3}$ is symmetric positive definite and there exists $\mu > 0$ such that
\[
\|A - \mu \text{Id}\|_{\infty} \le \mathcal{C}_{\Omega},
\]
where $\mathcal{C}_{\Omega}$ is a positive constant only depends on the domain $\Omega$. Then it holds that:
\[
\|w\|_{1,2,\Omega} \leq C \| A:D^{2}w \|_{\Omega},
\]
for all $w \in H^1_0(\Omega)$.
\end{proposition}
\begin{proof}
By triangle inequality, for any $\mu > 0$ we have
\begin{align*}
\|\mu \Delta w\|_{p,\Omega} & \le \|A:D^2 w - \mu \Delta w\|_{p,\Omega} + \|A:D^2 w\|_{p,\Omega} 
 \le \|A - \mu \text{Id}\|_{\infty} \|w\|_{2,p,\Omega} +  \|A:D^2 w\|_{p,\Omega}.
\end{align*}
If $\Omega \in \mathbb{R}^3$ be a polyhedral domain, in \cite{Douge92} it shows that for any $p \in [\frac65, \frac43)$, there exists $C_p > 0$ such that
\[
\|w\|_{2,p,\Omega} \le C_p \|\Delta w\|_{p,\Omega}.
\]
Inserting this inequality in the previous step we have
\[
C_p^{-1} \mu \|w\|_{2,p,\Omega} \le \|A - \mu \text{Id}\|_{\infty} \|w\|_{2,p,\Omega} +  \|A:D^2 w\|_{p,\Omega}.
\]
If there exists a $\mu>0$ such that $\|A - \mu \text{Id}\|_{\infty} < \frac12 C_p^{-1} \mu \|w\|_{2,p,\Omega}$, we can have a bound:
\[
\|w\|_{2,p,\Omega} \le  \frac{2C_p}{\mu} \|A:D^2 w\|_{p,\Omega}.
\]
Finally, by the Sobolev embedding, if $p = \frac65$ with $d=3$ we have
\[
\|w\|_{1,2,\Omega} \le C \|w\|_{2,p,\Omega} \le  \mathcal{C}_{\Omega} \|A:D^2 w\|_{p,\Omega}.
\]
Here $\mathcal{C}_\Omega$ only depends on the domain $\Omega$. This completes the proof.
\end{proof}

Next we present the first error estimate of the scheme \eqref{mini_h} which is summarized as follows:

\begin{theorem}
\label{main_error}
Let $(u,\Bq)$ be the exact solution of problem \eqref{non_div} that satisfies the regularity assumption \eqref{regularity_NonDiv} and $u \in H^{2+\delta}(\Omega)$ with 
$\delta > 0$. If $(u_h, \Bq_h)$ is a minimizer of \eqref{mini_h}, then 
\begin{equation}\label{main_result1}
\|u - u_h\|_{1,2,\Ct_h} +\|\Bq - \Bq_h\|_{\Ct_h} 
\le C \mathcal{C}_f
 h^{\min\{k-1,\delta\}}\|u\|_{\min\{k+1,2+\delta\},2,\Omega}.
\end{equation}
In addition, if the stronger regularity assumption \eqref{W2p_assumption} holds for the non-divergence operator, then we have the optimal convergence with respect to the $H^2$ norm:
\begin{equation}\label{main_result2}
\|u - u_h\|_{2,2,\Ct_h} 
\le C \mathcal{C}_f
 h^{\min\{k-1,\delta\}}\|u\|_{\min\{k+1,2+\delta\},2,\Omega}.
\end{equation}
Here $\mathcal{C}_f := \max ( \tau^{-2}\Vert f\Vert_{\Omega}^{2}, 1)$.
\end{theorem}

\begin{remark}
Here we only assume the exact solution has the strong $H^{2+\delta}$ regularity. The linear non-divergence operator 
is only assumed to have a weak $H^1$ regularity (\ref{regularity_NonDiv}) for the energy norm error estimate. In addition if we have full $H^2$ regularity for the nondivergence operator, we then obtain the optimal convergence under the discrete $H^2$ norm. 
\end{remark}

Moreover, we can establish the error estimate with the well-known Aleksandrov-Bakelman-Pucci 
estimates (ABP) \cite[Chapter 9, Theorem 9.1]{GilbargTrudinger1983}: 
If $w \in C^0(\bar{\Omega}) \cap W^{2,d}_{loc}(\Omega)\cap H_{0}^{1}(\Omega)$, then it holds:
\begin{equation}
\label{ABP_est}
\| w \|_{\infty, \Omega} \leq C \| A:D^{2}w \|_{d, \Omega}.
\end{equation}
The above estimate holds if the domain $\Omega \in \mathbb{R}^d$ $(d=2,3)$ is an open and bounded polytope,  and the coefficient matrix 
$A = A(x) \in [L^{\infty}(\Omega)]^{d\times d}$ is a symmetric positive definite matrix with eigenvalues 
bounded between $(0, \Lambda)$ with $\Lambda > 0$.

\begin{theorem}
\label{Linf_error}
Let $(u,\Bq)$ be the exact solution of the problem \eqref{non_div} and $(u_h, \Bq_h)$ be a  minimizer of 
\eqref{mini_h}. If $u \in H^{2+\delta}(\Omega)$ with $\delta > 0$, then  the following statements hold:

(i) In the 2D case with 
any polynomial degree $k \ge 2$ we have:
\begin{align*}
\|u - u_h\|_{\infty, \Omega} &\le C  \mathcal{C}_f
h^{\min\{k-1,\delta \}}\|u\|_{\min\{k+1,2+\delta\},2,\Omega}.
\end{align*}

(ii) In the 3D case with $k=2, 3$, $\delta > \frac{1}{2}$ and $A$ is piecewise polynomial with respect to $\Ct_h$, 
we have:
\[
\|u - u_h\|_{\infty, \Omega} \le C  \mathcal{C}_f
h^{\min\{k-1,\delta\} -\frac12} \|u\|_{\min\{k+1,2+\delta\},2,\Omega}.
\]
\end{theorem}

\begin{remark}
In the 3D case, we make two additional restrictions, one  on the the polynomial degree $k = 2,3$ and the other on the coefficient matrix $A$. This is due to the use of the discrete inverse inequality in the analysis. Nevertheless, for the 2D case, this requirement is not needed.
\end{remark}

\section{Proofs of the error estimates}
This section is devoted to providing the proof of the error estimates stated in the last section. We begin by gathering several auxiliary results to be used in the proof.

\begin{lemma}\label{aux_1}
If $(u_h,\Bq_h)$ is a minimizer of $\mathcal{J}_h$ with $\tau > 0$, then
\[
\tau \Cs(u_h, \Bq_h) \le \|f\|^2_{\Omega}
\]
and
\[
 \tau ( h^{-2}\|\lb \Bq_h \cdot \Bn \rb\|^2_{1,\Cf^I_h} + h^{-4}\|\lb u_h \rb\|^2_{1,\Cf_h}) \le 2 \|f\|^2_{\Omega} \mathcal{S}(u_h, \Bq_h).
\]
\end{lemma}
\begin{proof}
Since $(u_h,\Bq_h)$ is a minimizer of $\mathcal{J}_h$, we observe that
\[
\tau \Cs(u_h, \Bq_h) \le \mathcal{J}_h(u_h,\Bq_h) \le  \mathcal{J}_h(0,{\bf{0}}) = \|f\|^2_{\Omega}.
\]
The second inequality is a direct consequence of the first assertion.
\end{proof}

For discrete functions, we frequently use the following inverse and trace inequalities \cite{BrennerScott08}: For $w_h \in W(D)$ with $D \in \mathbb{R}^d$, we have that
\begin{align}\label{inverse_1}
\|w_h\|_{l,p,D} &\le C h_K^{m-l + \frac{d}{p} - \frac{d}{q}} \|w_h\|_{m,q,D}, \quad \text{for all $K \in \Ct_h$,}   \quad 0 \le m \le l, 1 \le p,q \le \infty, \\
\label{trace_1}
\|w_h\|_{\partial K} &\le C h^{-\frac12}_K \|w_h\|_K, \quad \text{for all $K \in \Ct_h$.} 
\intertext{The global version of the above inequalities are as follows: for $w_h \in W_h$ it holds}
\label{inverse_3}
\|w_h\|_{l,p,\Cf_h} &\le C h^{m-l + \min\{0, \frac{d}{p} - \frac{d}{q}\}} \|w_h\|_{m,q,\Cf_h}, \quad \text{for} \quad 0 \le m \le l, 1 \le p,q \le \infty,\\
\label{trace_2}
\|w_h\|_{\Cf_h} &\le C h^{-\frac12} \|w_h\|_{\Ct_h},\\
\label{trace_3}
\|\lb w_h \rb \|_{\Cf_h} &\le C h^{-\frac12} \|w_h\|_{\Ct_h}.
\end{align}

{{In the error estimates, the essential step is to bridging the discontinuous numerical solution $(u_h,\Bq_h) \in W_h \times \BV_h$ with the exact solution $(u,\Bq) \in H^1(\Omega) \times H(\text{div};\Omega)$ so that we can apply the regularity of the PDE \eqref{regularity_NonDiv} on the error. To this end, we use several lifting operators on discrete functions as follows.}}

The first lifting operator is to construct a $u^c_h \in W^c_h$ such that their difference is controlled by the jumps of $u_h$.  Namely, for an arbitrary $w_h \in W_h$, we can use the averaging operator introduced in \cite{KarakashianPascal07} to define $w^c_h$ such that $w^c_h \in W^c_h$ and it satisfies the property that \cite{KarakashianPascal07}[Theorem 2.1]:
\begin{equation*}
\|w_h - w^c_h\|^2_{\Ct_h} + h^2 \|\nabla (w_h - w^c_h)\|^2_{\Ct_h} \le C h\|\lb w_h \rb\|^2_{\Cf_h}.
\end{equation*}
With the inverse inequality \eqref{inverse_3} for discrete functions, we have the following result:
\begin{equation}\label{lifting_u}
\|w_h - w^c_h\|^2_{\Ct_h} + h^2 \|\nabla (w_h - w^c_h)\|^2_{\Ct_h} \le C \|\lb w_h \rb\|^2_{1,\Cf_h}.
\end{equation}
Likewise, for vector-valued functions $\Bv_h \in \BV_h$, we can construct a $H(\text{div})$ conforming function $\Bv_h^c \in \BV_h \cap H(\text{div};\Omega)$ such that the following approximation property is true:
 \begin{equation}\label{lifting_q}
\|\Bv_h - \Bv^c_h\|^2_{\Ct_h} + h^2 \|\nabla (\Bv_h - \Bv^c_h)\|^2_{\Ct_h} \le C \|\lb \Bv_h \cdot \Bn \rb\|^2_{1,\Cf^I_h}.
\end{equation}

{{Notice that $u^c_h \in W^c_h$ is only in $H^1(\Omega)$ which is not sufficient for the regularity condition $w \in H^1_0(\Omega), D^2 w \in L^2(\Omega)$. Hence, we need further lift the function to a function in $H^2(\Omega) \cap H^1_0(\Omega)$. We will need two different lifting operators for the $H^1$ and $L^{\infty}$ error estimates respectively.}}

The first one is due to \cite{BrennerSung2019}. They established the following result:

\begin{lemma}\label{aux_2}
For any function $w^c_h \in W^c_h$, there exists $\widetilde{w}^c_h \in H^{2}(\Omega) \cap H^1_0(\Omega)$ such that:
\[
\|D^l (w^c_h - \widetilde{w}^c_h)\|^2_{\Omega} \le C h^{3-2l} \|\lb \nabla w^c_h \cdot \Bn \rb\|^2_{\Cf^I_h}, \quad l=0,1,2.
\]
\end{lemma}

It is worth to mention that here the construction of $\widetilde{w}^c_h$ is done in the context of conforming virtual element spaces. This leads to the fact that $\widetilde{w}^c_h$ in general is not a piecewise polynomial in $W_h$. Nevertheless, our analysis for the first error estimate does not require the lifted function in the discrete space $W_h$. 

The next auxiliary result is on the relation between $(w_h, \Bv_h)$ and its conforming counterpart $(w^c_h, \Bv^c_h)$.
\begin{lemma}\label{aux_3}
If $(w^c_h, \Bv^c_h) \in W^c_h \times \BV^c_h$ is the lifting of  $(w_h, \Bv_h) \in W_h \times \BV_h$, then 
\begin{align*}
\|\lb \nabla w^c_h \cdot \Bn \rb\|^2_{\Cf_h^I} &\le C h^{-1}\|\Bv^c_h - \nabla w^c_h\|^2_{\Ct_h}, \\
\|\Bv^c_h - \nabla w^c_h\|^2_{\Ct_h} &\le C(\|\Bv_h - \nabla w_h\|^2_{\Ct_h} +  \|\lb \Bv_h \cdot \Bn \rb\|^2_{1,\Cf^I_h} + h^{-2}\|\lb w_h \rb\|^2_{1,\Cf_h}).
\end{align*} 
\end{lemma}
\begin{proof}
The first assertion is due to the fact that $\Bv_h^c \in \BV^c_h$ and the trace inequality \eqref{trace_3}:
\[
\|\lb \nabla w^c_h \cdot \Bn \rb\|^2_{\Cf_h^I} = \|\lb (\Bv^c_h - \nabla w^c_h) \cdot \Bn \rb\|^2_{\Cf_h^I} \le C h^{-1} \| (\Bv^c_h - \nabla w^c_h) \|^2_{\Ct_h}.
\]
For the second assertion we begin with triangle inequality:
\begin{align*}
\|\Bv^c_h - \nabla w^c_h\|^2_{\Ct_h} &\le C (\|\Bv_h - \nabla w_h\|^2_{\Ct_h} + \|\Bv^c_h - \Bv_h \|^2_{\Ct_h} + \|\nabla (w^c_h - w_h) \|^2_{\Ct_h}) \\
& \le C(\|\Bv_h - \nabla w_h\|^2_{\Ct_h} +  \|\lb \Bv_h \cdot \Bn \rb\|^2_{1,\Cf^I_h} + h^{-2}\|\lb w_h \rb\|^2_{1,\Cf_h}), 
\end{align*}
the last step is due to the approximation property of the lifting operators \eqref{lifting_u}, \eqref{lifting_q}.
\end{proof}

We are ready to present the proof of  Theorem \ref{main_error}:
\begin{proof}
We begin by proving 
\begin{equation}
\label{aux_4}
\|u - u_h\|^2_{1,2,\Ct_h} \le C  \max ( \tau^{-2}\Vert f\Vert_{\Omega}^{2}, 1) \mathcal{J}_h(u_h, \Bq_h).
\end{equation}
To this end, we let $(u^c_h, \Bq^c_h) \in W^c_h \times \BV^c_h$ be the lifting of $(u_h, \Bq_h)$. In addition, we let $\widetilde{u}^c_h \in \widetilde{W}^c_{k+2}$ be the $C^1$-lifting of $u^c_h$, and set
$$
T_1:= C\|u - \widetilde{u}^c_h\|^2_{1,2,\Ct_h}, \  T_2:=C\|\widetilde{u}^c_h - u^c_h\|^2_{1,2,\Ct_h}, \ 
T_3:=C\|u^c_h - u_h\|^2_{1,2,\Ct_h}.
$$
We use the triangle inequality to obtain
\begin{align*}
\|u - u_h\|^2_{1,2,\Ct_h} \le T_1 + T_2 + T_3.
\end{align*}
By the approximation property of the lifting operator \eqref{lifting_u}, we have that
\[
T_3 \le C h^{-2} \|\lb u_h \rb\|^2_{1,\Cf_h}.
\]
Similarly, by Lemma \ref{aux_2}, Lemma \ref{aux_3} and inverse inequality \eqref{inverse_3}, we can bound $T_2$ as:
\begin{align*}
T_2 &\le C h \|\lb \nabla {u}^c_h \cdot \Bn \rb\|^2_{\Cf^I_h} \le C\|\Bq^c_h - \nabla u^c_h\|^2_{\Ct_h} \\
&\le C(\|\Bq_h - \nabla u_h\|^2_{\Ct_h} +  \|\lb \Bq_h \cdot \Bn \rb\|^2_{1,\Cf^I_h} + h^{-2}\|\lb u_h \rb\|^2_{1,\Cf_h})
\end{align*}
For $T_1$, we recall the regularity assumption \eqref{regularity_NonDiv} to have that
\begin{equation}\label{XXX}
    T_1 \le C \|A:D^2 \widetilde{u}^c_h - A:D^2 u\|^2_{\Omega} = C \|A:D^2 \widetilde{u}^c_h - f\|^2_{\Omega}.
\end{equation}
Introducing the notation
$$
T_{11}:=\|A:(\nabla \Bq_h - \nabla \Bq^c_h)\|^2_{\Ct_h}, \
T_{12}:=\|A:(\nabla \Bq^c_h - D^2 u^c_h)\|^2_{\Ct_h}, 
$$
and
$$
T_{13}:= \|A:(D^2 u^c_h- D^2 \widetilde{u}^c_h)\|^2_{\Ct_h}),
$$
by the triangle inequality, we further split the right hand side of equation \eqref{XXX} as
\begin{align*}
T_1 \le  C (\|A: \nabla \Bq_h - f\|^2_{\Ct_h} + T_{11} + T_{12} + T_{13}).
\end{align*}
For $T_{11}$, by the lifting property \eqref{lifting_q} we have:
\[
T_{11} \le C \|\nabla (\Bq_h - \Bq^c_h)\|^2_{\Ct_h} \le h^{-2} \|\lb \Bq_h \cdot \Bn \rb\|^2_{1,\Ce^I_h}.
\]
For $T_{12}$, with inverse inequality \eqref{inverse_3}, we have:
\begin{align*}
T_{12} &\le C \| \nabla \Bq^c_h - D^2 u^c_h\|^2_{\Ct_h} \le C h^{-2} \| \Bq^c_h - \nabla u^c_h\|^2_{\Ct_h} \\
\intertext{then with Lemma \ref{aux_3}, we further have:}
T_{12} & \le C(h^{-2}\|\Bq_h - \nabla u_h\|^2_{\Ct_h} + h^{-2} \|\lb \Bq_h \cdot \Bn \rb\|^2_{1,\Cf^I_h} + h^{-4}\|\lb u_h \rb\|^2_{1,\Cf_h})
\end{align*}
For $T_{13}$, by virtue of Lemma \ref{aux_2} we have:
\[
T_{13} \le C h^{-1} \| \lb \nabla u^c_h \cdot \Bn \rb \|^2_{\Ce_h^I}.
\]
We can bound the right side by Lemma \ref{aux_3} to have:
\[
T_{13} \le C h^{-2} (\|\Bq_h - \nabla u_h\|^2_{\Ct_h} +  \|\lb \Bq_h \cdot \Bn \rb\|^2_{1,\Cf^I_h} + h^{-2}\|\lb u_h \rb\|^2_{1,\Cf_h})
\]
Now if we combine the above estimates for $T_1, T_2, T_3$ we have:
\[
\|u - u_h\|^2_{1,2,\Ct_h} \le C ( \|A: \nabla \Bq_h - f\|^2_{\Ct_h} + h^{-2}\|\Bq_h - \nabla u_h\|^2_{\Ct_h} + h^{-2} \|\lb \Bq_h \cdot \Bn \rb\|^2_{1,\Cf^I_h} + h^{-4}\|\lb u_h \rb\|^2_{1,\Cf_h}).
\]
By virtue of Lemma \ref{aux_1}, we have:
\[
\|u - u_h\|^2_{1,2,\Ct_h} \le C ( \|A: \nabla \Bq_h - f\|^2_{\Ct_h} + h^{-2}\|\Bq_h - \nabla u_h\|^2_{\Ct_h} 
+ 2\tau^{-1} \|f\|^2_{\Omega} \Cs(u_h,\Bq_h)).
\]
We notice that 
\begin{align*}
& \|A: \nabla \Bq_h - f\|^2_{\Ct_h} + h^{-2}\|\Bq_h - \nabla u_h\|^2_{\Ct_h} + 2\tau^{-1} \|f\|^2_{\Omega} \Cs(u_h,\Bq_h) \\
\leq & \|A: \nabla \Bq_h - f\|^2_{\Ct_h} + h^{-2}\|\Bq_h - \nabla u_h\|^2_{\Ct_h} 
+ \max ( \dfrac{2\Vert f\Vert_{\Omega}^{2}}{\tau^{2}} , 1) \tau \Cs(u_h,\Bq_h) \\
\leq & \max ( \dfrac{2\Vert f\Vert_{\Omega}^{2}}{\tau^{2}} , 1) 
( \|A: \nabla \Bq_h - f\|^2_{\Ct_h} + h^{-2}\|\Bq_h - \nabla u_h\|^2_{\Ct_h} + \tau \Cs(u_h,\Bq_h)).
\end{align*}
Then we have:
\begin{align*}
 \|u - u_h\|^2_{1,2,\Ct_h} 
\le & C \mathcal{C}_f
( \|A: \nabla \Bq_h - f\|^2_{\Ct_h} + h^{-2}\|\Bq_h - \nabla u_h\|^2_{\Ct_h} + \tau \Cs(u_h,\Bq_h)) \\
= & C \mathcal{C}_f \Cj_h(u_h,\Bq_h).
\end{align*}
This completes the proof of \eqref{aux_4}. Finally, we recall that $(u_h, \Bq_h)$ is the minimizer of the functional $\Cj_h$ \eqref{mini_h}. This implies that:
\begin{align*}
\|u - u_h\|^2_{1,2,\Ct_h} \le C  \mathcal{C}_f \Cj_h(u_h,\Bq_h) 
\le C \mathcal{C}_f \Cj_h(w_h,\Bv_h)
\end{align*}
for all $(w_h, \Bv_h) \in W_h \times \BV_h$. Let $\Pi_W u$ be the Lagrange interpolant of $u$ over $W^c_h$ and $\boldsymbol{\Pi}_V \Bq$ be the {\bf{BDM}} projection of $\Bq$ in $\BV_h^c$. By taking $(w_h,\Bv_h) = (\Pi_W u,\boldsymbol{\Pi}_V \Bq)$ we have that $\Cs(\Pi_W u,\boldsymbol{\Pi}_V \Bq) = 0$ since they are in the comforming subspaces $W^c_h \times \BV^c_h$. The above estimate implies:
\begin{align*}
& \|u - u_h\|^2_{1,2,\Ct_h} \le C \mathcal{C}_f \Cj_h(u_h,\Bq_h) 
\le C \mathcal{C}_f \Cj_h(\Pi_W u,\boldsymbol{\Pi}_V \Bq) \\
= & C \mathcal{C}_f (\|A:\nabla \boldsymbol{\Pi}_V \Bq - f\|^2_{\Ct_h} 
+ h^{-2}\|\boldsymbol{\Pi}_V \Bq - \nabla \Pi_W u\|^2_{\Ct_h} + \tau \Cs(\Pi_W u,\boldsymbol{\Pi}_V \Bq) \\
= & C  \mathcal{C}_f (\|A:\nabla (\boldsymbol{\Pi}_V \Bq - \Bq)\|^2_{\Ct_h} 
+ h^{-2}\|(\boldsymbol{\Pi}_V \Bq - \Bq) - \nabla (\Pi_W u - u)\|^2_{\Ct_h} ) \\
\le & C  \mathcal{C}_f h^{2\min\{k-1, \delta\}} 
\|u\|^2_{\min\{k+1, 2+ \delta\},2,\Omega}.
\end{align*}
The last step we used the standard approximation properties of the interpolations $(\Pi_W u,\boldsymbol{\Pi}_V \Bq)$. This completes the proof of the error estimates for $\|u - u_h\|_{1,2,\Ct_h}$. For the error in $\Bq$, notice that \eqref{aux_4} implies that:
\[
\|u - u_h\|^2_{1,2,\Ct_h} + h^{-2}\|\Bq_h - \nabla u_h\|^2_{\Ct_h} 
\le C  \mathcal{C}_f \mathcal{J}_h(u_h, \Bq_h).
\]
With the same argument as the above, we have:
\[
h^{-1}\|\Bq_h - \nabla u_h\|_{\Ct_h} \le C  \mathcal{C}_f
h^{\min\{k-1, \delta\}} \|u\|_{\min\{k+1, 2+ \delta\},2,\Omega}.
\]
By triangle inequality, we have:
\[
\|\Bq - \Bq_h\|_{\Ct_h} \le \|\Bq_h - \nabla u_h\|_{\Ct_h} + \|\nabla u - \nabla u_h\|_{\Ct_h} 
\le C  \mathcal{C}_f h^{\min\{k-1, \delta\}} \|u\|_{\min\{k+1, 2+ \delta\},2,\Omega}.
\]
This completes the proof for \eqref{main_result1}.

Finally, if we assume the full $H^2$ regularity \eqref{W2p_assumption}, the above proofs can be seamlessly modified to get the optimal convergence rate for the discrete $H^2$ norm. Namely, we can use the same splitting:
\begin{align*}
\|u - u_h\|^2_{2,2,\Ct_h} &\le C(\|u - \widetilde{u}^c_h\|^2_{2,2,\Ct_h} + \|\widetilde{u}^c_h - u^c_h\|^2_{2,2,\Ct_h} + \|u^c_h - u_h\|^2_{2,2,\Ct_h}) \\
& \le \widetilde{T}_1 + C h^{-2}T_2 +  Ch^{-2}T_3.
\end{align*}
Here we applied the standard dicsrete inverse inequality \eqref{inverse_1}. This extra $h^{-2}$ power is not going to degenerate the order of convergence since in the above estimate $T_1$ is the dominating error while $T_2, T_3$ are $h^2$ higher order terms.

For $\widetilde{T}_1$, by virtue of the full $H^2$ regularity \eqref{W2p_assumption}, we can bound this term as:
\[
\widetilde{T}_1 \le C \|A:D^2 \widetilde{u}^c_h - A:D^2 u\|^2_{\Omega} = C \|A:D^2 \widetilde{u}^c_h - f\|^2_{\Omega}
\]
and the rest estimate is exactly the same as $T_1$. Therefore, the final convergence rate remains the same as the $H^1$ norm without losing any order. This completes the proof of \eqref{main_result2}.

\end{proof}

For the second error estimates Theorem \ref{Linf_error}, the main difference from the above proof lays in the fact that we relax the regularity assumption to \eqref{ABP_est}:
\[
\| w \|_{\infty, \Omega} \leq C \| A:D^{2}w \|_{d, \Omega}.
\]
For $d=3$, note that the RHS of the above ABP estimates become a stronger $L^3$ norm of the source term. This forces us to use a discrete inverse inequality in the analysis. Consequently, we need to use another lifting operator such that the resulting function remains to be piecewise polynomials. Namely,  for any function $w_h \in W^c_h$, let $\widetilde{W}^c_{h}$ to be the classical $C^1$-conforming Hsieh-Clough-Tocher macroelement \cite{DouglasDupont79} of degree $k+2$. In \cite{Neilan19} they prove that for $k \ge 2$ with $d = 2$ and $k=2,3$ with $d = 3$, there exists a $\widehat{w}^c_h \in \widetilde{W}^c_{h}$ such that the same approximation property as in Lemma \ref{aux_2} holds:
\[
\|D^l (w^c_h - \widetilde{w}^c_h)\|^2_{\Omega} \le C h^{3-2l} \|\lb \nabla w^c_h \cdot \Bn \rb\|^2_{\Cf^I_h}, \quad l=0,1,2.
\]

It is worth to mention that the use of virtual element lifting in Lemma \ref{aux_2} is necessary since the above Hsieh-Clough-Tocher macroelement lifting is only defined for $k=2,3$ in 3D. 

Next we define the $H^2$ conforming projection of $u$ in two steps: 

\begin{definition}
Let $u \in H^{2+\delta}(\Omega) \cap H^1_0(\Omega)$. We use $\Pi_h u \in W^c_h$ to denote the standard Lagrange interpolant of $u$. $\widetilde{\Pi}_h^c u$ is the $H^2$ projection of $\Pi^c_h u$ by using the Hsieh-Clough-Tocher macroelement lifting. 
\end{definition}

We have the following approxmation property of $\widetilde{\Pi}_h^c u$:
\begin{lemma}\label{aux_u}
Let $u \in H^{2+\delta}(\Omega) \cap H^1_0(\Omega)$ and $\widetilde{\Pi}_h^c u \in \widetilde{W}^c_h \cap H^1_0(\Omega)$ be the $H^2$ interpolant of $u$ defined as above. We have:
\[
\|u - \widetilde{\Pi}_h^c u\|_{l,2,\Omega} \le C h^{\min\{k+1, 2 + \delta\} - l} \|u\|_{\min \{ k+1, 2+\delta \}, 2, \Omega}, \quad l = 0,1,2.
\]
\end{lemma}

\begin{proof}
By the triangle inequality, standard interpolant property and Lemma \ref{aux_2} we have:
\begin{align*}
\|u - \widetilde{\Pi}_h^c u\|_{l,2,\Omega} & \le \|u - \Pi_h^c u\|_{l,2,\Omega}  + \|\Pi_h^c u - \widetilde{\Pi}_h^c u\|_{l,2,\Omega} \\
& \le C h^{\min\{k+1, 2 + \delta\} - l} \|u\|_{\min \{ k+1, 2+\delta \}, 2, \Omega}  + C h^{\frac32 - l} \| \lb \nabla \Pi^c_h u \cdot \Bn \rb\|_{\Cf^I_h}.
\end{align*}
The second term on the right hand side can be written as:
\begin{align*}
C h^{\frac32 - l} \| \lb \nabla \Pi^c_h u \cdot \Bn \rb\|_{\Cf^I_h} &= C h^{\frac32 - l} \| \lb \nabla (u - \Pi^c_h u) \cdot \Bn \rb\|_{\Cf^I_h}\\
&\le C h^{\min\{k+1, 2 + \delta\} - l} \|u\|_{\min \{ k+1, 2+\delta \}, 2, \Omega}.
\end{align*}
The last step is due to the trace inequality and the approximation property of Lagrange interpolant. This completes the proof.
\end{proof}

\begin{proof}{\bf of Theorem \ref{Linf_error}}
The ABP regularity \eqref{ABP_est} now reads as:
\begin{equation}
\label{ABP_2}
\|w\|_{\infty,\Omega} \le C \| A:D^{2}w \|_{\Omega}, \quad \forall w \in C^{0}(\overline{\Omega}) 
\cap W_{loc}^{2,2}(\Omega) \cap H_{0}^{1}(\Omega).  
\end{equation}
Similar as in the above proof, we split the error into three terms:
\begin{align}
\label{ABP_split}
\|u - u_h\|^2_{\infty, \Omega} &\le C(\|u - \widetilde{u}^c_h\|^2_{\infty, \Omega} 
+ \|\widetilde{u}^c_h - u^c_h\|^2_{\infty, \Omega} + \|u^c_h - u_h\|^2_{\infty, \Omega}) \\
\nonumber
& := T_1 + T_2 + T_3.
\end{align}
For $T_2, T_3$, we can use the same arguments as in the above proof to obtain:
\begin{align*}
T_2 & \le C h^{-d} \|\widetilde{u}^c_h - u^c_h\|^2_{\Ct_{h}}
 \le C h^{-2}(\|\Bq_h - \nabla u_h\|^2_{\Ct_h} +  \|\lb \Bq_h \cdot \Bn \rb\|^2_{1,\Cf^I_h} 
+ h^{-2}\|\lb u_h \rb\|^2_{1,\Cf_h}), \\
T_3 & \le C h^{-d} \|u^c_h - u_h\|^2_{\Ct_{h}}
\le C h^{-4} \|\lb u_h \rb\|^2_{1,\Cf_h}.
\end{align*}
For $T_1$, we further split it by inserting the projection of $u$ defined in Lemma \ref{aux_u}:
\begin{align*}
T_1 & \le C (\|u - \widetilde{\Pi}^c_h u\|^2_{\infty, \Omega} 
+ \|\widetilde{\Pi}^c_h u - \widetilde{u}^c_h\|^2_{\infty, \Omega}) \\
& \le C (\|u - \widetilde{\Pi}^c_h u\|^2_{2, 2, \Omega} 
+ \|\widetilde{\Pi}^c_h u - \widetilde{u}^c_h\|^2_{\infty, \Omega}). 
\end{align*}

We first consider the case $d = 2$.
Since $\widetilde{\Pi}^c_h u - \widetilde{u}^c_h \in \widetilde{W}^c_h \cap H^1_0(\Omega) \subset C^0(\bar{\Omega}) 
\cap W^{2,2}_{loc}(\Omega)$, we can use the ABP condition \eqref{ABP_2} to obtain:
\begin{align*}
 \|\widetilde{\Pi}^c_h u - \widetilde{u}^c_h\|^2_{\infty, \Omega} 
 \le &C \|A:D^2(\widetilde{\Pi}^c_h u - \widetilde{u}^c_h)\|^2_{\Omega} \\
\le & C \| A: D^2 (u - \widetilde{\Pi}^c_h u) \|^2_{\Omega} + C \|A:D^2\widetilde{u}^c_h - f\|^2_{\Omega}.
\end{align*}
So overall we have:
\begin{align*}
T_1 &\le C (\|u - \widetilde{\Pi}^c_h u\|^2_{2,2,\Omega} + \| A: D^2 (u - \widetilde{\Pi}^c_h u) \|^2_{\Omega} 
+ \|A:D^2\widetilde{u}^c_h - f\|^2_{\Omega}) \\
&\le C h^{2\min\{k-1, \delta\}} \|u\|^2_{\min \{ k+1, 2+\delta \}, 2, \Omega} + C \|A:D^2\widetilde{u}^c_h - f\|^2_{\Omega}.
\end{align*}
At this point, the second term on the right hand side in the above estimate is the same as in the estimates of $T_1$ in the proof of Theorem 2.1. Therefore, we can proceed in the same way from here to have the final estimates as:
\begin{align*}
  \|u - u_h\|^2_{\infty, \Omega}
\le & C h^{2\min\{k-1, \delta\}} \|u\|^2_{\min \{ k+1, 2+\delta \}, 2, \Omega} 
+ C  \mathcal{C}_f \mathcal{J}_h(u_h, \Bq_h) \\
\le & C  \mathcal{C}_f
h^{2\min\{k-1, \delta\}} \|u\|^2_{\min\{k+1, 2+ \delta\},2,\Omega},
\end{align*}
if the exact solution is smooth enough $u \in H^{2+\delta}(\Omega)$ for some $\delta > 0$.

Finally we consider the case $d = 3$. The burden for the sacrifice in the convergence rate lays in the ABP estimates. 
Namely, now the ABP estimate \eqref{ABP_est} becomes:
\begin{equation}
\|w\|_{\infty,\Omega} \le C \| A:D^{2}w \|_{3,\Omega}, \quad \forall w \in C^{0}(\overline{\Omega}) 
\cap W_{\text{loc}}^{2,3}(\Omega) \cap H_{0}^{1}(\Omega).
\end{equation}

The $L^3$ norm on the right side is stronger than the one we have in the minimizer. We bypass this issue with standard discrete inverse inequality which degenerates the order of convergence in the final stage. We lay out the details of the proof below.

Begin with the spitting as in \eqref{ABP_split} and $T_1$. The only term needs modification is $\|\widetilde{\Pi}^c_h u - \widetilde{u}^c_h\|^2_{\infty, \Omega}$. More precisely, we can apply the ABP estimate as:
\[
C \|\widetilde{\Pi}^c_h u - \widetilde{u}^c_h\|^2_{\infty, \Omega} 
\le C \|A:D^2(\widetilde{\Pi}^c_h u - \widetilde{u}^c_h)\|^2_{3, \Omega}.
\]
Notice that if $A$ is piecewise polynomial, the last term in the above estimate $A:D^2(\widetilde{\Pi}^c_h u - \widetilde{u}^c_h)$ remains to be piecewise polynomials over $\Ct_h$, we can now apply the inverse inequality \eqref{inverse_1} to obtain:
\begin{align*}
\|A:D^2(\widetilde{\Pi}^c_h u - \widetilde{u}^c_h)\|^2_{3, \Omega} &\le C h^{-1} \|A:D^2(\widetilde{\Pi}^c_h u - \widetilde{u}^c_h)\|^2_{\Omega} \\ 
& \le C h^{-1} (\|A:D^2(\widetilde{\Pi}^c_h u - u)\|^2_{\Omega} + \|A:D^2 \widetilde{u}^c_h - f\|^2_{\Omega}) 
\end{align*}
From here we can use the same argument to control the last term on the right hand side as in the previous case. Finally, due to the $h^{-1}$ from the inverse inequality, we have the final estimates as:
\begin{align*}
\|u - u_h\|^2_{\Ct_h} &\le C h^{2\min\{k-1, \delta\} - 1} \|u\|^2_{\min \{ k+1, 2+\delta \}, 2, \Omega} 
+ C  \mathcal{C}_f h^{-1} \mathcal{J}_h(u_h, \Bq_h) \\ 
&\le C  \mathcal{C}_f h^{2\min\{k-1, \delta\} -1} \|u\|^2_{\min\{k+1, 2+ \delta\},2,\Omega},
\end{align*}
if the exact solution is smooth enough $u \in H^{2+\delta}(\Omega)$ for some $\delta > \frac{1}{2}$.
This completes the proof.

\end{proof}

\section{Fixed-Point Proximity Algorithms}

In this section, we develop a fast numerical algorithm for optimization problem \eqref{mini_h} based on the Fixed-Point Proximity Algorithms (\FPPA{}) \cite{micchelli2011proximity,li2015multi,RenInexact} and establish a convergence theorem for the algorithm.
    
A challenge of solving problem \eqref{mini_h} is that the $L_1$ norm involved in its objective function is non-differentiable, making classical gradient-based methods invalid. To surmount this difficulty, we first characterize solutions of the optimization problem \eqref{mini_h} as fixed-points of a nonlinear map involved \textit{proximity operators} of functions appearing in the objective function. Based on the fixed-point equation, application of a matrix splitting technique to the fixed-point equation leads to the \FPPA{} with guaranteed convergence.

\subsection{An Equivalent Formulation of the Basic Optimization Problem.}
We first reformulate the optimization problem \eqref{mini_h} in a form convenient to apply \FPPA{}. To this end, we choose a basis of the finite dimensional space $W_h \times \BV_h$ and denote it by $\mathcal{B}_{W_h \times V_h}$. The numerical solution $(u_h,\Bq_h)$ is represented by its coefficients $x$ with respect to this basis. Note the choice of the basis is crucial to obtain a sparse representation of the solution. We will demonstrate this point with two choices of bases: a standard Lagrange nodal basis and a {\em multiscale basis}.

We associate each element in $(w,\Bv)\in W_h\times\BV_h$ a vector $x\in\bR^N$, which represents its coefficients in the basis. This association is one-to-one. Hence,  $W_h \times \BV_h$ is isomorphic to $\bR^N$.
Specifically, we use $\{w_i\}_{i=1}^{N_W}$ to denote the basis of $W_h$ and $\{\Bv_i\}_{i=1}^{N_V}$ the basis of $\BV_h$. Let
\begin{align*}
    B_{11}&:=h^{-2}\jz{\kh{\nabla w_i,\nabla w_j}_{\Ct_h}}_{1\leq i,j\leq N_W},\\
    B_{12}&:=h^{-2}\jz{\kh{\nabla w_i,\Bv_j}_{\Ct_h}}_{1\leq i\leq N_W,1\leq j\leq N_V},\\
    B_{22}&:=\jz{\kh{A:\nabla\Bv_i,A:\nabla\Bv_j}_{\Ct_h}+h^{-2}\kh{\Bv_i,\Bv_j}_{\Ct_h}}_{1\leq i,j\leq N_V},
\end{align*}
and define
$$
    B:=\jz{B_{11}&B_{12}\\B_{12}^\ast&B_{22}}.
$$
We also define
$$
 b:=\jz{\kh{-2\kh{A:\nabla\Bv_j,f}_{\Ct_h}}_{1\leq j\leq N_V}\\0_{N_W}},
$$
where $0_{N_W}$ denotes the $N_W$-dimensional zero vector.
In terms of the symmetric positive semi-definite matrix $B$ and the vector $b$ defined above, 
we write 
$$
\normm{A:\nabla\Bv-f}_{\Ct_h}^2+h^{-2}\normm{\Bv+\nabla w}_{\Ct_h}^2-\normm{f}^2_{\Ct_h}=x^\ast Bx+b^\ast x.
$$
Likewise, we rewrite the boundary terms in  \eqref{mini_h}. To this end, 
we introduce matrix $L\in\bR^{M\times N}$ by
\begin{equation}\label{eq:L}
    L:=\jz{\tau h^{-1}L_1\\\tau h^{-1}L_2\\\tau h^{-2}L_3}
\end{equation}
with $L_1\in\bR^{2|\Ce_h^I|\times N}$, $L_2\in\bR^{3|\Ce_h^I|\times N}$ and $L_3\in\bR^{3|\Ce_h^B|\times N}$. Here, $L_1x$ contains all endpoint values of $\zzkh{\Bv \cdot \Bn}$ on $\Ce_h^I$, $L_2x$ contains all endpoint and midpoint values of $\zzkh{w}$ on $\Ce_h^I$ and $L_3x$ contains all endpoint and midpoint values of $w$ on $\Ce_h^B$.
We then introduce $d\in\bR^M$ by setting
\[
    d :=\jz{0_{2|\Ce_h^I|}\\0_{3|\Ce_h^I|}\\d_1},
\]
where $d_1\in\bR^{_{2|\Ce_h^B|}}$ contains all endpoint and midpoint values of $g$ on $\Ce_h^B$, with each element arranged according to $L_3x$.
In terms of the matrix $L$ and the vector $d$, we identify
$$
{\tau h^{-1} \normm{\zzkh{\Bv \cdot \Bn}}_{1,\Ce^I_h}+\tau h^{-1}\normm{\zzkh{w} }_{1,\Ce^I_h}+\tau h^{-2}\normm{w-g}_{1,\Ce^B_h}}=\normm{Lx-d}_1.
$$
In this notation, problem \eqref{mini_h} is reformulated as
\begin{equation}\label{eq:mdl_0}
        \argmin_{x\in\bR^N}\dkh{x^\ast Bx+b^\ast x+\normm{Lx-d}_1}.
\end{equation}
    For convenience of further analysis, we define functions $g:\bR^N\to\bR$ by $$
    g(x)\coloneqq x^\ast Bx+b^\ast x, \ \ \mbox{for}\ \  x\in\bR^N
    $$ 
    and $h:\bR^M\to\bR$ by 
    $$
    h(y)\coloneqq\|y-d\|_1,\ \ \mbox{for}\ \  y\in\bR^M.
    $$
    With these definitions, we identify \eqref{eq:mdl_0} as
    \begin{equation}\label{eq:mdl}
        \argmin_{x\in\bR^N}\dkh{g\kh{x}+h\kh{Lx}}.
    \end{equation}
    Clearly, \eqref{eq:mdl} is a convex optimization problem studied in \cite{li2015multi, micchelli2011proximity}. Since the objective function of  \eqref{eq:mdl} is non-differentiable, the classical gradient-based method is not applicable.

\subsection{A Fixed-Point Formulation.}
    We follow \cite{li2015multi, micchelli2011proximity} to reformulate \eqref{eq:mdl} as a fixed-point problem.
    %
    %
   To this end, we recall some terminologies of convex analysis \cite{micchelli2011proximity, rockafellar2009variational}. We denote by $\Gamma_0(\bR^n)$ the class of all lower semi-continuous and proper convex functions mapping $\bR^N\to\bR\cup\{+\infty\}$. Notice that $g\in\Gamma_0(\bR^N)$ and $h\in\Gamma_0(\bR^M)$. By $\nj{\cdot}{\cdot}$ we denote the Euclidean inner product on $\bR^N$. For any function $f\in\Gamma_0(\bR^N)$, we define the convex conjugate of $f$ by
    \[
        f^\ast\kh{u}\coloneqq\sup_{x\in\bR^N}\dkh{\nj{x}{u}-f\kh{u}}.
    \]
    Let $\bS_{+}^{N}\subset\bR^{N\times N}$ denote the collection of all symmetric positive definite matrices. The proximity operator of $f\in\Gamma_0(\bR^N)$ with respect to $P\in\bS_{+}^{N}$ is a function mapping $\bR^N\to\bR^N$, defined by
    \begin{equation}\label{eq:prox}
        \prox_{f,P}\kh{x}\coloneqq\argmin_{u\in\bR^N}\dkh{f\kh{u}+\frac{1}{2}\normm{u-x}_P^2},\qquad \text{for all }x\in\bR^N,
    \end{equation}
    where $\|{\cdot}\|_P^2\coloneqq\nj{\cdot}{P\cdot}$. It is worth remarking that for $f\in\Gamma_0(\bR^N)$, we always have $f^\ast\in\Gamma_0(\bR^N)$, and there holds
    \begin{equation}\label{eq:conj_prox}
        \prox_{f^\ast,P}=\Id-P^{-1}\circ\prox_{f,P^{-1}}\circ P.
    \end{equation}
    By a direct application of relation between proximity operator and subdifferntial of functions in $\Gamma_0(\bR^N)$, we can easily derive the following characterization theorem \cite{li2015multi, micchelli2011proximity}.
    
    \begin{theorem}
        If $x_\ast\in\bR^N$ is a solution of \eqref{eq:mdl}, then there exists $y_\ast\in\bR^M$ such that for all $\alpha>0$, $P\in\bS^N_{+}$ and $Q\in\bS^M_{+}$, there holds
        \begin{equation}\label{eq:fpe}
            \left\{
                \begin{aligned}
                    x_\ast&=\prox_{\alpha g,P}\kh{x_\ast-P^{-1}L^\ast y_\ast},\\
                    y_\ast&=\prox_{(\alpha h)^\ast,Q}\kh{Q^{-1}Lx_\ast+y_\ast}.
                \end{aligned}
            \right.
        \end{equation}
        Conversely, if there exist $P\in\bS^N_{+}$ and $Q\in\bS^M_{+}$ satisfying the fixed-point equations \eqref{eq:fpe}, then the corresponding $x_\ast$ is a solution of \eqref{eq:mdl}.
    \end{theorem}

    It is worth noticing that due to the smoothness of function $g$, there is an alternative form of fixed-point characterization of solutions of \eqref{eq:mdl} involving gradient of $g$. This gradient approach simplifies the fixed-point equation \Cref{eq:fpe}, but at the same time requires exact estimation of Lipschitz constant of the gradient for convergence analysis, which is complicated for problem \cref{mini_h}. To have a clear convergence condition, we choose \cref{eq:fpe} for further discuss. For details of the gradient approach, readers are referred to \cite{Li2015Fast}.
    
    We can rewrite fixed-point equation \eqref{eq:fpe} in a compact form.  For $\alpha>0$, $P\in\bS_+^N$ and $Q\in\bR^M_+$, we define block matrices
    \[ R\coloneqq \jz{P&0\\0&Q},\quad E\coloneqq\jz{I&-P^{-1}L^\ast\\Q^{-1}L&I},
    \]
    and function
    \[
        \Phi\kh{z}\coloneqq \alpha g\kh{x}+\kh{\alpha h\kh{y}}^\ast.
    \]
    Then,  \eqref{eq:fpe} is equivalent to
    \begin{equation}\label{eq:cmp_fpe}
        z_\ast=\prox_{\Phi,R}\kh{Ez_\ast},
    \end{equation}
    where $z_\ast\coloneqq\kh{x_\ast,y_\ast}\in\bR^{N+M}$. 
    That is, $z_\ast$ is a fixed-point of operator $\prox_{\Phi,R}\circ E$.
    
    The fixed-point characterization \eqref{eq:cmp_fpe} brings two advantages. First of all, this formulation provides a guidance for designing convergence-guaranteed algorithms. Moreover, by this characterization, the non-differentiability of $\|\cdot\|_1$ is no longer an obstacle, as long as the proximity operator $\prox_{\Phi,R}$ can be easily calculated. Indeed, as we will show next, operator $\prox_{\Phi,R}$ has a closed-form. 
    
    
    \begin{pps}\label{pps:prox}
    If $\alpha>0$, $P\in\bS_{+}^N$ and $Q:={\rm diag}(q_i)_{i=1}^M\in\bS_{+}^M$, then for any $x\in\bR^N$ and $y\in\bR^M$, 
        \[
            \prox_{\alpha g,P}\kh{x}=\kh{P+2\alpha B}^{-1}\kh{Px-\alpha b},
        \]
        and
        \[
            \prox_{(\alpha h)^\ast,Q}\kh{y}=\zkh{\Proj_{\zkh{-\alpha,\alpha}}\kh{y_i-d_i/q_{i}}}_{i=1}^M,
        \]
        where $\Proj_{\zkh{-\alpha,\alpha}}$ denotes the projection operator onto $\zkh{-\alpha,\alpha}$.
    \end{pps}
    \begin{proof}
        Both of these closed-form formulas can be shown by the definition of the proximity operator. By \eqref{eq:prox},
        \[
            \prox_{\alpha g,P}\kh{x}=\argmin_{u\in\bR^N}\dkh{\alpha u^\ast Bu+\alpha b^\ast u+\frac{1}{2}\normm{u-x}^2_P},
        \]
        whose right-hand side is a minimization of a quadratic function of $u$. Taking the gradient of the right-hand side and solving for the zero gradient, we obtain the first formula.
        
        It remains to establish the formula for $\prox_{(\alpha h)^\ast,Q}$. By identity \eqref{eq:conj_prox}, for any $y\in\bR^M$ we have that
        \[
            \prox_{(\alpha h)^\ast,Q}\kh{y}=y-Q^{-1}\prox_{\alpha h,Q^{-1}}\kh{Qy}.
        \]
        By setting $\tilde v\coloneqq v-d$, we observe that
        \begin{align*}
            \prox_{\alpha h,Q^{-1}}\kh{Qy}&=\argmin_{v\in\bR^M}\dkh{\alpha\normm{v-d}_1+\frac{1}{2}\normm{v-Qy}^2_{Q^{-1}}}\\
            &=d+\argmin_{\tilde v\in\bR^M}\dkh{\alpha\normm{\tilde v}_1+\frac{1}{2}\normm{\tilde v-\kh{Qy-d}}^2_{Q^{-1}}}\\
            &=d+\prox_{\alpha\normm{\cdot}_1,Q^{-1}}\kh{Qy-d}.
        \end{align*}
        Since $Q$ is diagonal, we obtain that  
        \[
            \prox_{\alpha\normm{\cdot}_1,Q^{-1}}\kh{Qy-d}=\zkh{\kh{\kh{y_i-\alpha}q_i-d_i}_+\sgn\kh{\kh{y_i-\alpha}q_i-d_i}}_{i=1}^M,
        \]
        where $(\cdot)_+\coloneqq\max\{\cdot,0\}$ and $\sgn$ is the sign function (for details, see \cite{micchelli2011proximity}). Finally, combining the equations above yields the desired formula.
    \end{proof}

    With help of \Cref{pps:prox}, we can easily obtain the closed form of $\prox_{\Phi,R}$ by the following calculation. For any $z=\kh{x,y}\in\bR^{N+M}$, by definition of proximity operator \cref{eq:prox}, we have
    \begin{align*}
        \prox_{\Phi,R}\kh{z}&=\argmin_{u\in\bR^N,v\in\bR^M}\dkh{\alpha g\kh{u}+\kh{\alpha h}^\ast\kh{v}+\frac{1}{2}\normm{u-x}^2_P+\frac{1}{2}\normm{v-y}^2_Q}\\
        &=\kh{\prox_{\alpha g,P}\kh{x},\prox_{\kh{\alpha h}^\ast,Q}\kh{y}}.
    \end{align*}
    
    \subsection{Fixed-Point Proximity Algorithms}
    
    In this section we consider the numerical algorithm of solving fixed-point equation \eqref{eq:cmp_fpe}. The na\"ive Picard's iteration might not converge, due to the expanding nature of the linear mapping $E$ \cite{li2015multi,Jin2022Inexact,RenInexact}. Fortunately a matrix-splitting technique will do the trick. Specifically, we split $E$ into
    \[
        E=\kh{E-M}+M,
    \]
    then the Fixed-Point Proximity Algorithm (\FPPA{}) is proposed by
    \begin{equation}\label{alg:FPPA}
        z_{k+1}=\prox_{\Phi,R}\kh{\kh{E-M}z_{k+1}+Mz_k}.
    \end{equation}
    Based on properties of proximity operators, a corollary of \cite{Jin2022Inexact,RenInexact} provides the following convergence theorem. Define $\bS^{N}$ as the collection of all positively semi-definite matrices on $\bR^{N\times N}$. 
    
    \begin{theorem}\label{thm:conv}
        Suppose $\alpha>0$. If $RM\in\bS^{N+M}$, then for any initial $z_0\in\bR^{N+M}$, the \FPPA{} scheme \eqref{alg:FPPA} converges to a solution of fixed-point equation \eqref{eq:cmp_fpe}.
    \end{theorem}
    
    Notice that \Cref{thm:conv} extends results from \cite{li2015multi}, where $RM\in\bS^{N+M}_+$ is required. Although \eqref{alg:FPPA} is generally an implicit algorithm, if $E-M$ is chosen to be strictly block-triangular, then \eqref{alg:FPPA} is actually explicit. We set
    \[
        M\coloneqq\jz{I&P^{-1}L^\ast\\Q^{-1}L&I}
    \]
    in \eqref{alg:FPPA}, then the explicit \FPPA{} for \eqref{eq:mdl_0} follows as \Cref{alg:ex_FPPA}.
    \begin{algorithm}
        \caption{Explicit \FPPA{} for \eqref{eq:mdl_0}}\label{alg:ex_FPPA}
        \begin{algorithmic}[1]
            \REQUIRE $\alpha>0$, $P\in\bS_+^N$, $\{q_i\}_{i=1}^M\subset(0,+\infty)$, $x_0\in\bR^N$ and $y_0\in\bR^M$.
            \STATE $k\leftarrow 0$
            \REPEAT
                \STATE $y_{k+1}\leftarrow[\Proj_{\zkh{-\alpha,\alpha}}\kh{(y_k)_i+((Lx_k)_i-d_i)/q_i}]_{i=1}^M$
                \STATE $x_{k+1}\leftarrow\kh{P+2\alpha B}^{-1}\kh{L^\ast\kh{y_k-2y_{k+1}}+Px_k-\alpha b}$
                \STATE $k\leftarrow k+1$
            \UNTIL{stop criteria is met}
            \ENSURE $x_\ast\leftarrow x_{k}$
        \end{algorithmic}
    \end{algorithm}
    
\subsection{Convergence Analysis.}
    This section provides convergence theorem for \FPPA{} \cref{alg:ex_FPPA} based on \Cref{thm:conv}, and discusses one-dimension case for illustration propose. A direct application of \Cref{thm:conv} gives the following convergence result of \Cref{alg:ex_FPPA}.
    
    \begin{theorem}\label{thm:FPPA_conv}
        Suppose $\alpha>0$, $P\in\bS_{+}^N$ and $Q\in\bS_{+}^M$ is diagonal. If $\normm{Q^{-1/2}LP^{-1/2}}\leq1$, then for any $x_0\in\bR^N$ and $y_0\in\bR^M$, explicit \FPPA{} \Cref{alg:ex_FPPA} generates a sequence $\dkh{x_k}_{k\in\bN}$ converging to a solution of \eqref{eq:mdl_0}.
    \end{theorem}
    \begin{proof}
        By setting $R\coloneqq{\rm Diag}(P,Q)$ and
        \[
            M\coloneqq\jz{I&P^{-1}L^\ast\\Q^{-1}L&I},
        \]
        in \eqref{alg:FPPA}, we have \FPPA{} as an explicit iteration as
        \[
            \left\{
                \begin{aligned}
                    y_{k+1}&=\prox_{(\alpha h)^\ast,Q}\kh{y_k+Q^{-1}Lx_k},\\
                    x_{k+1}&=\prox_{\alpha g,P}\kh{P^{-1}L^\ast\kh{y_k-2y_{k+1}}+x_k}.
                \end{aligned}
            \right.
        \]
        The above iteration exactly fits into \Cref{alg:ex_FPPA} by applying \Cref{pps:prox}. It is direct to check that $RM\in\bS$ if and only if $\|Q^{-1/2}LP^{-1/2}\|\leq1$. Therefore conditions of \Cref{thm:conv} obtain, and the convergence directly follows.
    \end{proof}
    Here we discuss the simple case of \Cref{thm:FPPA_conv} with $d=1$ and $g=0$ to illustrate \Cref{thm:FPPA_conv} in greater detail. In this case the second-order elliptic equation reduces to the following homogeneous boundary value problem
    \begin{align}
        A:u'' &= f, \quad \text{in} \quad \zkh{0,1}, 
        \label{eq:1d_pde_1}\\
        u & = 0, \quad \text{on} \quad \dkh{0,1}. 
        \label{eq:1d_pde_2}
    \end{align}
    The $\Ct_h$ consists of intervals $\{I_i\coloneqq[ih,(i+1)h]\}_{i=0}^{1/h-1}$, and interior faces $\Ce_h^I=\dkh{ih}_{i=1}^{1/h-1}$ and boundary faces $\Ce_h^B=\dkh{0,1}$. Therefore $W_h$ and $\BV_h$ are linear spaces of piecewise quadratic and linear functions on $\Ct_h$ respectively.
    
    If we choose Lagrangian bases for $W_h$ and $\BV_h$, then each element $(w,\Bv)\in W_h\times\BV_h$ is represented by vector $(x_w,x_v)\in\bR^{3/h}\times\bR^{2/h}$, where $((x_w)_{3i},(x_w)_{3i+1},(x_w)_{3i+2})$ are function values of $w$ at points $(ih,(i+1/2)h,(i+1)h)$ on element $I_i$, and $((x_v)_{2i},(x_v)_{2i+1})$ are function values of $\Bv$ at points $(ih,(i+1)h)$ on element $I_i$. Then by \eqref{eq:L}, we have $L_1=[0_{(1/h-1)\times(3/h)},L_{12}]$, $L_2=[L_{21},0_{(1/h-1)\times(2/h)}]$ and $L_3=[L_{31},0_{2\times (2/h)}]$, where
    \begin{align}
        L_{12}&=\jz{
            0&1&-1&0&0&\cdots&0&0&0\\
            0&0&0&1&-1&\cdots&0&0&0\\
            \vdots&\vdots&\vdots&\vdots&\vdots&\ddots&\vdots&\vdots\\
            0&0&0&0&0&\cdots&1&-1&0
        }\in\bR^{(1/h-1)\times(2/h)},\label{eq:L11}\\
        L_{21}&=\jz{
            0&0&1&-1&0&0&0&\cdots&0&0&0&0\\
            0&0&0&0&0&1&-1&\cdots&0&0&0&0\\
            \vdots&\vdots&\vdots&\vdots&\vdots&\vdots&\vdots&\ddots&\vdots&\vdots&\vdots&\vdots\\
            0&0&0&0&0&0&0&\cdots&1&-1&0&0
        }\in\bR^{(1/h-1)\times(3/h)},\label{eq:L22}\\
        L_{31}&=\jz{1&0&0&\cdots&0&0&0\\
        0&0&0&\cdots&0&0&1}\in\bR^{2\times(3/h)}.\label{eq:L32}
    \end{align}
    Therefore we have the following corollary of \Cref{thm:FPPA_conv} for problem \Cref{eq:1d_pde_1,eq:1d_pde_2}.
    \begin{proposition}\label{thm:1d_fppa}
        Suppose $\alpha,\lambda>0$, $\{q_i\}_{i=1}^{2/h}\subset(0,+\infty)$ and $P=\lambda I_{5/h}$, $Q={\rm Diag}(q_i)_{i=1}^{2/h}$, $L$ is defined as in \cref{eq:L11,eq:L22,eq:L32}. If
        \begin{equation}\label{eq:1d}
            \lambda\geq\frac{\tau^2}{h^4}\max\dkh{2h^2\max\dkh{q_i^{-1}}_{i=1}^{2/h-2},q_{2/h-1}^{-1},q_{2/h}^{-1}},
        \end{equation}
        then for any $x_0\in\bR^{5/h}$ and $y_0\in\bR^{2/h}$, explicit \FPPA{} \Cref{alg:ex_FPPA} generates a sequence $\dkh{x_k}_{k\in\bN}$ converging to a solution of \eqref{eq:mdl_0}.
    \end{proposition}
    \begin{proof}
        By \Cref{thm:FPPA_conv}, it is sufficient to prove $\|Q^{1/2}LP^{-1/2}\|<1$. Observe that
        \[
            LL^\ast=\frac{\tau^2}{h^2}{\rm Diag}\kh{L_{11}L_{11}^\ast,L_{22}L_{22}^\ast,h^{-2}L_{32}L_{32}^\ast}=\frac{\tau^2}{h^4}{\rm Diag}\kh{2h^2I_{2/h-2},I_2}.
        \]
        Then by condition, we have
        \begin{align*}
            \normm{Q^{-1/2}LP^{-1/2}}^2&=\normm{Q^{-1/2}LP^{-1}L^\ast Q^{-1/2}}\\
            &=\frac{\tau^2}{\lambda h^4}\max\dkh{2h^2\max\dkh{q_i^{-1}}_{i=1}^{2/h-2},q_{2/h-1}^{-1},q_{2/h}^{-1}}\leq1.
        \end{align*}
        Finally, by \Cref{thm:FPPA_conv}, we finish the proof.
    \end{proof}

    \Cref{thm:1d_fppa} extends results from \cite{li2015multi} as previously discussed, where the strict inequality is required in \eqref{eq:1d}. \Cref{thm:1d_fppa} also implies that for all given $h$, $P=\lambda I$ and diagonal $Q\in\bS_+$, \FPPA{} \Cref{alg:ex_FPPA} always converges with sufficiently large $\lambda>0$. Similar result can be easily obtained for general $P\in\bS_+$ with sufficiently large $\|P\|$. For case $d=2,3$, the matrix $L$ depends on the adjacency matrix. Nevertheless there are always similar structures as \cref{eq:L11,eq:L22,eq:L32} as long as Lagrangian bases are picked for $W_n$ and $\BV_h$.
    
\section{Numerical Results}

In this section we provide numerical experiments for \cref{non_div_governing} validating the error estimates obtained in this work via explicit \FPPA{} method. Specially, we test $k=2$ in numerical experiments, with regular solution on square domain, or singular solution on L-shaped domain. In each case we test different type of coefficients $A$, i.e.\ constant, continuous and discontinuous coefficients. In \Cref{ssec:S,ssec:L} we choose the Lagrangian basis for $W_h$ and $\BV_h$ and show the order of various errors as $N=1/h$ doubles. In \Cref{ssec:M} we show examples of solving \cref{non_div_governing} with multiscale DG bases. We demonstrate the numerical solution naturally possess a sparse representation with such choice of basis.

\subsection{Numerical Results for Various Coefficients on Square Domain}\label{ssec:S}

We first consider problem \eqref{non_div} on domain $[0,1]\times[0,1]$ with constant, continuous and discontinuous coefficient matrices $A$, respectively. Whatever coefficient matrix $A$ is, the right-hand side functions $f$ and $g$ in \eqref{non_div} are chosen such that the exact solution is always
\[
    u\kh{x,y}=\sin\kh{\pi x}\sin\kh{\pi y}.
\]

\subsubsection{Constant Coefficients}

Here we choose
\[
    A=\jz{2&1\\1&2}
\]
in \eqref{non_div}. \Cref{tab:SS_A1} shows various error measures and convergence rates.
\begin{table}[!ht]
    \centering
    \begin{tabular}{c|cccccccc}
        $N$ & $\normm{u-u_h}_{2}$ & order & $\norm{u-u_h}_{1,2}$ & order & $\norm{u-u_h}_{2,2}$ & order & $\normm{\Bq-\Bq_h}_{0,2}$ & order  \\\hline
4 & 4.92e$-$02 & -- & 2.58e$-$01 & -- & 3.52e+00 & -- & 2.59e$-$01 & -- \\
8 & 1.42e$-$02 & 1.80& 7.49e$-$02 & 1.78& 1.76e+00 & 1.00& 7.59e$-$02 & 1.77\\
16 & 3.80e$-$03 & 1.90& 2.01e$-$02 & 1.90& 8.61e$-$01 & 1.04& 2.02e$-$02 & 1.91\\
32 & 9.75e$-$04 & 1.96& 5.16e$-$03 & 1.96& 4.23e$-$01 & 1.02& 5.14e$-$03 & 1.98\\
64 & 2.46e$-$04 & 1.99& 1.30e$-$03 & 1.99& 2.10e$-$01 & 1.01& 1.29e$-$03 & 2.00
    \end{tabular}
    \caption{Various errors and convergence rates for regular solution with constant coefficients.}
    \label{tab:SS_A1}
\end{table}

\subsubsection{Continuous Coefficients}

Here we choose
\[
    A=\jz{1+x&\sqrt{xy}\\\sqrt{xy}&1+y}
\]
in \eqref{non_div}. \Cref{tab:SS_A2} shows various error measures and convergence rates.
\begin{table}[!ht]
    \centering
    \begin{tabular}{c|cccccccc}
        $N$ & $\normm{u-u_h}_{2}$ & order & $\norm{u-u_h}_{1,2}$ & order & $\norm{u-u_h}_{2,2}$ & order & $\normm{\Bq-\Bq_h}_{2}$ & order  \\\hline
4 & 4.48e$-$02 & -- & 2.44e$-$01 & -- & 3.55e+00 & -- & 2.51e$-$01 & -- \\
8 & 1.21e$-$02 & 1.89& 6.59e$-$02 & 1.89& 1.75e+00 & 1.02& 6.68e$-$02 & 1.91\\
16 & 3.18e$-$03 & 1.93& 1.72e$-$02 & 1.94& 8.56e$-$01 & 1.03& 1.72e$-$02 & 1.96\\
32 & 8.08e$-$04 & 1.98& 4.36e$-$03 & 1.98& 4.23e$-$01 & 1.02& 4.30e$-$03 & 2.00\\
64 & 2.02e$-$04 & 2.00& 1.09e$-$03 & 2.00& 2.10e$-$01 & 1.01& 1.07e$-$03 & 2.01
    \end{tabular}
    \caption{Various errors and convergence rates for regular solution with continuous coefficients.}
    \label{tab:SS_A2}
\end{table}

\subsubsection{Discontinuous Coefficients}

Here we choose
\[
    A=\jz{2&{\rm sgn}\kh{\kh{x-0.5}\kh{y-0.5}}\\{\rm sgn}\kh{\kh{x-0.5}\kh{y-0.5}}&2}
\]
in \eqref{non_div}. \Cref{tab:SS_A3} shows various error measures and convergence rates.
\begin{table}[!ht]
    \centering
    \begin{tabular}{c|cccccccc}
        $N$ & $\normm{u-u_h}_{2}$ & order & $\norm{u-u_h}_{1,2}$ & order & $\norm{u-u_h}_{2,2}$ & order & $\normm{\Bq-\Bq_h}_{2}$ & order  \\\hline
4 & 4.89e$-$02 & -- & 2.56e$-$01 & -- & 3.50e+00 & -- & 2.52e$-$01 & -- \\
8 & 1.33e$-$02 & 1.88& 7.04e$-$02 & 1.86& 1.75e+00 & 1.00& 7.02e$-$02 & 1.84\\
16 & 3.58e$-$03 & 1.89& 1.88e$-$02 & 1.90& 8.57e$-$01 & 1.03& 1.88e$-$02 & 1.90\\
32 & 9.38e$-$04 & 1.93& 4.88e$-$03 & 1.95& 4.23e$-$01 & 1.02& 4.84e$-$03 & 1.96\\
64 & 2.40e$-$04 & 1.96& 1.24e$-$03 & 1.97& 2.10e$-$01 & 1.01& 1.23e$-$03 & 1.98
    \end{tabular}
    \caption{Various errors and convergence rates for regular solution with discontinuous coefficients.}
    \label{tab:SS_A3}
\end{table}

\subsection{Numerical Results for Various Coefficients on L-Shaped Domain}\label{ssec:L}
We then consider singular solution of \eqref{non_div} on L-shaped domain  with constant, continuous and discontinuous coefficient matrices $A$, respectively. Whatever coefficient matrix $A$ is, the right-hand side functions $f$ and $g$ in \eqref{non_div} are chosen such that the exact solution is always
\[
    u\kh{r,\theta}=r^{2/3}\sin\kh{2\theta/3}
\]
over the L-shaped domain $\zkh{-1,1}\times\zkh{-1,1}\setminus(0,1]\times[-1,0)$.

\subsubsection{Constant Coefficients}

Here we choose
\[
    A=\jz{1&0\\0&1}
\]
in \eqref{non_div}. \Cref{tab:SING_A1} shows various error measures and convergence rates.
\begin{table}[!ht]
    \centering
    \begin{tabular}{c|cccccccc}
        $N$ & $\normm{u-u_h}_{2}$ & order & $\norm{u-u_h}_{1,2}$ & order & $\norm{u-u_h}_{2,2}$ & order & $\normm{\Bq-\Bq_h}_{0,2}$ & order  \\\hline
4 & 9.32e$-$03 & -- & 7.93e$-$02 & -- & 1.08e+00 & -- & 7.99e$-$02 & -- \\
8 & 6.86e$-$03 & 0.44& 3.45e$-$02 & 1.20& 8.67e$-$01 & 0.32& 3.64e$-$02 & 1.13\\
16 & 5.29e$-$03 & 0.38& 1.99e$-$02 & 0.79& 7.08e$-$01 & 0.29& 2.08e$-$02 & 0.81\\
32 & 3.46e$-$03 & 0.61& 1.21e$-$02 & 0.72& 5.89e$-$01 & 0.27& 1.23e$-$02 & 0.75\\
64 & 2.34e$-$03 & 0.57& 7.80e$-$03 & 0.63& 4.85e$-$01 & 0.28& 7.88e$-$03 & 0.65
    \end{tabular}
    \caption{Various errors and convergence rates for singular solution with constant coefficients.}
    \label{tab:SING_A1}
\end{table}

\subsubsection{Continuous Coefficients}

Here we choose
\[
    A=\jz{1+\norm{x}&\sqrt{\norm{xy}}\\\sqrt{\norm{xy}}&1+\norm{y}}
\]
in \eqref{non_div}. \Cref{tab:SING_A2} shows various error measures and convergence rates.
\begin{table}[!ht]
    \centering
    \begin{tabular}{c|cccccccc}
        $N$ & $\normm{u-u_h}_{2}$ & order & $\norm{u-u_h}_{1,2}$ & order & $\norm{u-u_h}_{2,2}$ & order & $\normm{\Bq-\Bq_h}_{2}$ & order  \\\hline
4 & 1.06e$-$01 & -- & 8.47e$-$01 & -- & 5.02e+00 & -- & 8.50e$-$01 & -- \\
8 & 7.78e$-$03 & 3.76& 3.73e$-$02 & 4.51& 8.69e$-$01 & 2.53& 3.89e$-$02 & 4.45\\
16 & 6.31e$-$03 & 0.30& 2.28e$-$02 & 0.71& 7.08e$-$01 & 0.30& 2.36e$-$02 & 0.72\\
32 & 4.28e$-$03 & 0.56& 1.44e$-$02 & 0.66& 5.90e$-$01 & 0.26& 1.47e$-$02 & 0.68\\
64 & 3.04e$-$03 & 0.50& 9.85e$-$03 & 0.55& 4.88e$-$01 & 0.27& 9.93e$-$03 & 0.57
    \end{tabular}
    \caption{Various errors and convergence rates for singular solution with continuous coefficients.}
    \label{tab:SING_A2}
\end{table}

\subsubsection{Discontinuous Coefficients}

Here we choose
\[
    A=\jz{2&{\rm sgn}\kh{xy}\\{\rm sgn}\kh{xy}&2}
\]
in \eqref{non_div}. \Cref{tab:SING_A3} shows various error measures and convergence rates.
\begin{table}[!ht]
    \centering
    \begin{tabular}{c|cccccccc}
        $N$ & $\normm{u-u_h}_{2}$ & order & $\norm{u-u_h}_{1,2}$ & order & $\norm{u-u_h}_{2,2}$ & order & $\normm{\Bq-\Bq_h}_{2}$ & order  \\\hline
4 & 9.96e$-$03 & -- & 7.94e$-$02 & -- & 1.08e+00 & -- & 8.01e$-$02 & -- \\
8 & 4.76e$-$03 & 1.06& 3.16e$-$02 & 1.33& 8.67e$-$01 & 0.32& 3.39e$-$02 & 1.24\\
16 & 4.81e$-$03 & -0.01& 1.87e$-$02 & 0.76& 7.11e$-$01 & 0.29& 1.99e$-$02 & 0.77\\
32 & 4.17e$-$03 & 0.21& 1.41e$-$02 & 0.41& 5.96e$-$01 & 0.26& 1.44e$-$02 & 0.47\\
64 & 3.60e$-$03 & 0.21& 1.16e$-$02 & 0.27& 4.93e$-$01 & 0.27& 1.17e$-$02 & 0.30
    \end{tabular}
    \caption{Various errors and convergence rates for singular solution with discontinuous coefficients.}
    \label{tab:SING_A3}
\end{table}

\newpage

\subsection{Numerical Results Using a Multiscale Basis}\label{ssec:M}
In this section we perform examples of solving \cref{non_div_governing} using multiscale bases to illustrate the adaptivity/sparsity feature of $L^1$ stabilization. 
We test \cref{eq:1d_pde_1,eq:1d_pde_2} with following example
\begin{equation}\label{mdl:multi}
u(x) = \begin{cases} 
      t\kh{\kh{\frac{x}{t}}^{n+1}-\frac{x}{t}}, & 0 \leq x \leq t, \\
      \kh{1-t}\kh{\frac{1-x}{1-t}-\kh{\frac{1-x}{1-t}}^{n+1}}, & t< x\leq 1, \\
   \end{cases}
  \end{equation}
  where $0<t<1$ and $n>0$. The corresponding $q$ has a kink at $x=t$:
  \[
q(x) = u'(x) = \begin{cases} 
      \kh{n+1}\kh{\frac{x}{t}}^{n}-1,  & 0 \leq x \leq t, \\
      \kh{n+1}\kh{\frac{1-x}{1-t}}^{n}-1, & t< x\leq 1. \\
   \end{cases}
  \]
  
  Now we introduce the multiscale piecewise polynomial space and basis on $[0,1]$ once presented in \cite{Chen15}. Define $\bX_n\coloneqq W_{2^{-n}}$ for $n\in\bN$ and transformations of $f\in L^2\zkh{0,1}$ as
  \[
    \kh{\tau_0f}\kh{x}\coloneqq f\kh{2x}\chi_{\zkh{0,1}}\kh{2x},\qquad \kh{\tau_1f}\kh{x}\coloneqq f\kh{2x-1}\chi_{\zkh{0,1}}\kh{2x-1},
  \]
  where $\chi_{\zkh{0,1}}:\bR\to\bR$ takes $1$ on $\zkh{0,1}$ and $0$ otherwise. Then it is proven in \cite{Chen15} that for all $n\in\bN$,
  \[
    \bX_n=\bX_0\oplus^\perp\bW_1\oplus^\perp\cdots\oplus^\perp\bW_n,
  \]
  where $\bW_1$ is defined by $\bX_2=\bX_1\oplus^\perp\bW_1$ and for $n\geq 2$,
  \[
    \bW_n\coloneqq \tau_0\bW_{n-1}\oplus \tau_1\bW_{n-1}.
  \]
  Here `$\oplus^\perp$' indicates addition between two orthogonal sets in $L^2$.
  
  In \cite{Chen15} the multiscale piecewise quadratic polynomial basis is proposed as follows. There are three basis functions for $\bX_0$, as
  \[
    w_{00}\kh{x}\coloneqq1,\quad w_{01}\kh{x}\coloneqq\sqrt{3}\kh{2x-1},\quad w_{02}\kh{x}\coloneqq\sqrt{5}\kh{6x^2-6x+1},
  \]
  and for $\bW_1$,
  \begin{align*}
      w_{10}\kh{x}&\coloneqq\begin{cases}
        1-6x,&x\in\zkh{0,1/2},\\
        5-6x,&x\in(1/2,1],
      \end{cases}\\
      w_{11}\kh{x}&\coloneqq\begin{cases}
        \frac{\sqrt{91}}{31}\kh{240x^2-116x+9},&x\in\zkh{0,1/2},\\
        \frac{\sqrt{91}}{31}\kh{3-4x},&x\in(1/2,1],
      \end{cases}\\
      w_{12}\kh{x}&\coloneqq\begin{cases}
        \frac{\sqrt{91}}{31}\kh{4x-1},&x\in\zkh{0,1/2},\\
        \frac{\sqrt{91}}{31}\kh{240x^2-364x+133},&x\in(1/2,1].
      \end{cases}
  \end{align*}
  Define $M_1\coloneqq\dkh{w_{1i}}_{i=1}^3$ and $M_n\coloneqq \tau_0 M_{n-1}\cup\tau_1 M_{n-1}$ for all $n\geq 2$. Then \cite{Chen15} proves that $M_n$ is a basis of $\bW_n$ for $n\geq 1$. Such multiscale basis has vanishing moment of $4$ since by construction we have $M_n\perp\bX_0$, and this gives us a sparse representation of solutions. Finally the elements of representation are truncated under certain threshold.
  
  For problem \eqref{mdl:multi}, we set $t=1/6$ and $h=2^{-7}$, and the reconstructed function and coefficients of single-/multi-scale basis for reconstructed $u$ are shown in Figure \ref{fig:uni_plot}. Here we truncate all the multi-scale coefficient with absolute value under $10^{-5}$. One can see that the coefficients of single-scale basis are dense, while multiscale basis decade rapidly due to its vanishing moment, but have strokes around kink $x=1/6$. This shows the adaptivity/sparsity of the proposed multiscale basis comparing to single-scale basis. \Cref{tab:1d_error} quantifies the $L^2$ error and the coefficient sparsity of single-/multi-scale solutions with various truncation thresholds.
  
  \begin{figure}[ht]
      \centering
      \includegraphics[width=0.8\textwidth]{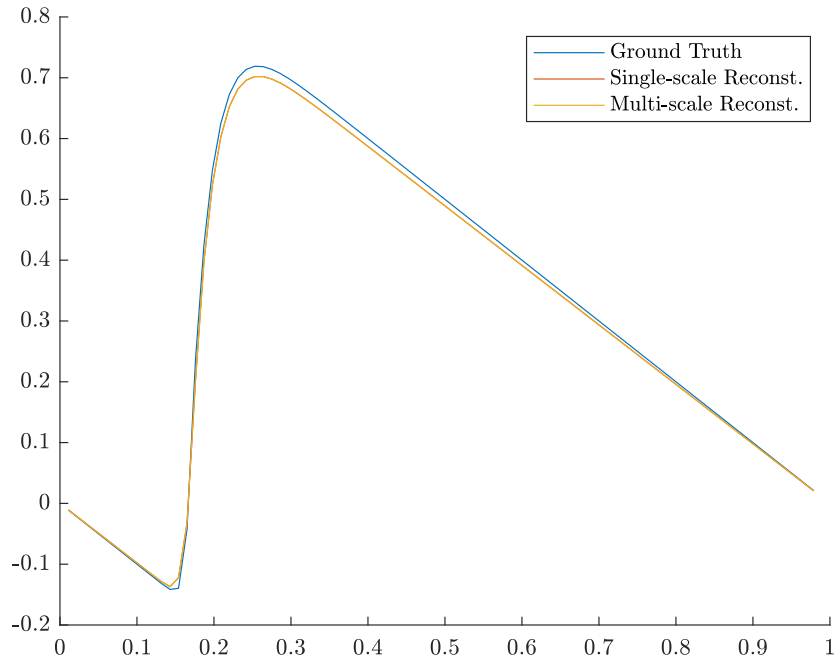}
      \caption{Ground truth and reconstructed solutions with single- and multi-scale bases. Single- and multi-scale solutions overlap in plot.}\label{fig:multi_sol}
  \end{figure}
  \begin{figure}[ht]
      \centering
      \includegraphics[width=0.8\textwidth]{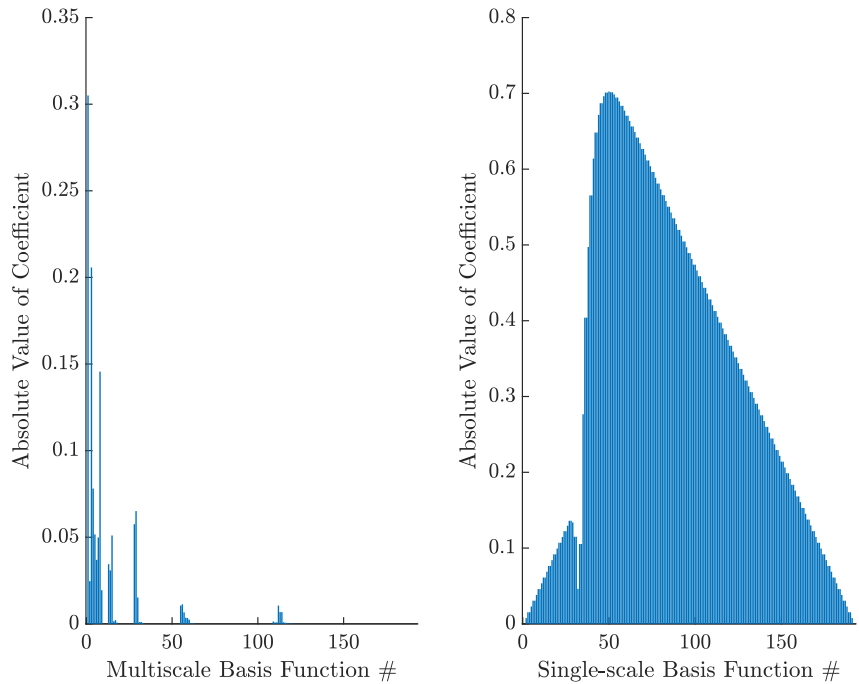}
      \caption{Coefficients comparison of multi- and single-scale bases.}\label{fig:uni_plot}
  \end{figure}
  \begin{table}[ht]
      \centering
      \begin{tabular}{c|ccc}
          Basis & Truncation Threshold & $\normm{u-u_h}_2$ & Coefficient Sparsity \\\hline
          Single-scale & -- & $1.097\times10^{-2}$ & $0\%$\\ \hline 
          \multirow{3}{*}{Multi-scale} 
          & $10^{-2}$ & $1.181\times10^{-2}$ & $90.63\%$\\
          & $10^{-3}$ & $1.099\times10^{-2}$ & $85.42\%$\\
          & $10^{-4}$ & $1.097\times 10^{-2}$ & $78.13\%$
      \end{tabular}
      \caption{The $L^2$-error and the coefficient sparsity of single- and multi-scale basis solutions with various truncation thresholds. Coefficient sparsity is defined as the percentage of zero elements.}
      \label{tab:1d_error}
  \end{table}

\section{Concluding Remark}

In this paper, we propose a novel $DG$ method for the second order elliptic equations in the non-divergence form with $L^1$ stabilization. The motivation for such an approach is to develop numerical schemes for PDEs such that the resulting solutions have sparse representation. This work is the first attempt along this path. Due to the use of $L^1$ stabilization in the minimization formulation, we are unable to use existing analytical tools for gradient based methods. We believe the analysis used in the error estimates and the 
\FPPA{} algorithm can be applied to more challenging problems with low regularity settings. It is worth to point out that the original $L^1$ regularization technique is different from what we proposed in this paper. Namely, the purpose of regularization is to relax the functional from ill-posedness. Meanwhile, in our work the $L^1$ term is a stablization term which ensures the convergence and existence of the minimizer of the discrete problem without breaking the consistency with the PDE. In turn, the sparsity promoting feature of the multiscale piecewise polynomial basis is not fully utilized in the current formulation. Nevertheless, we believe that a possible way to improve the performance of the scheme is to add a regularization term such as $\|Bx\|_1$ in this functional where $B$ is the multiscale piecewise polynomial decomposition. This is subject to our ongoing work in this project.   
 

\bibliographystyle{amsplain}
\bibliography{ref} 


\end{document}